\pdfoutput=1
\documentclass{article}
\usepackage{amsthm}
\usepackage{amsmath}
\usepackage{amssymb}
\usepackage{amscd}
\usepackage{graphicx}

\newtheorem{theorem}[subsection]{Theorem}
\newtheorem{lemma}[subsection]{Lemma}

\newtheorem{corollary}[subsection]{Corollary}

\theoremstyle{definition}

\newtheorem{definition}[subsection]{Definition}

\theoremstyle{remark}

\newtheorem{remark}[subsection]{Remark}

\def\R{\mathbb{R}}

\vspace{1.5cm}
\title{{\bf Newton flows for elliptic functions I}\\
{\small {\bf Structural stability:
Characterization \& Genericity }}}

\author{G.F. Helminck,\\
Korteweg-de Vries Institute\\
University of Amsterdam\\
P.O. Box 94248\\
1090 GE Amsterdam\\
The Netherlands\\
e-mail: g.f.helminck@uva.nl\\ 
F. Twilt,\\
Department of Applied Mathematics\\
University of Twente\\
P.O. Box 217, 7500 AE Enschede\\
The Netherlands\\
e-mail: f.twilt@kpnmail.nl\\
}
 
\begin{document}
\maketitle 

\begin{abstract}
\noindent
Newton flows are {\it dynamical systems} generated by a continuous, desingularized Newton method for mappings from a Euclidean space to itself. We focus on the special case of meromorphic functions on the complex plane. Inspired by the analogy between the rational (complex) and the elliptic (i.e., doubly periodic meromorphic) functions, a theory on the class of so-called {\it Elliptic Newton flows} is developed.

With respect to an appropriate topology on the set of all elliptic functions $f$ of fixed order $ r (\geqslant 2)$ we prove: For {\it almost all} functions $f$, the corresponding Newton flows are {\it structurally stable} i.e., topologically invariant under small perturbations of the zeros and poles for $f$ [{\it genericity}]. They can be described in terms of nondegeneracy-properties of $f$ similar to the rational case [{\it characterization}].
\end{abstract}



\section{Meromorphic Newton flows}
\label{sec1}

In this section we briefly explain the concept of meromorphic Newton flow. For details and historical notes, see 
\cite{Bra}, \cite{Gom}, \cite{JJT1}, \cite{JJT2}, \cite{Shub} and \cite{Smale85}.
In the sequel, let $f$ stand for a non-constant meromorphic function on the complex plane. So, $f(z)$ is complex analytic for all $z$ in $\mathbb{C}$ with the possible exception of (countably many) isolated singularities: the {\it poles} for $f$. 

The (damped) Newton method
for finding zeros of $f$ (with starting point $z^{0}$) is given by 
\begin{equation}
\label{vgl1}
z_{n+1}\!-\!z_{n}=-t_{n} \frac{f(z_{n})}{f^{'}(z_{n})}, \;t_{n} \neq 0, \; n=0,1,\hdots,
z_{0}=z^{0}.
\end{equation}
Dividing both sides of (\ref{vgl1}) by the ''damping factor'' $t_{n}$ and choosing $t_{n}$ smaller and smaller, yields an ''infinitesimal version'' 
of (\ref{vgl1}), namely
\begin{equation}
\label{vgl2}
  \dfrac{dz}{dt} = \dfrac{-f (z)}{f^{'} (z)}.
\end{equation}
Conversely, Euler's method applied to (\ref{vgl2}), gives rise to an iteration of the form (\ref{vgl1}).  
A dynamical system of type (\ref{vgl2}) is denoted by $\mathcal{N} (f)$. 
For this system we will interchangeably use the following terminologies: 
{\it vector field} [i.e.\;the expression on its r.h.s.], or {\it (Newton-)flow} [when we focus on its phase portrait(=family of all maximal trajectories as point sets)].

Obviously, zeros and poles for $f$ are removable singularities for $\frac{f}{f'}$ and turn into {\it isolated equilibria} for $\mathcal{N} (f)$. Special attention should be paid to those points $z$ where $f(z)\neq 0$ and $f^{'}(z) = 0$. In these (isolated!) so-called {\it critical points }, the vector field $\mathcal{N} (f)$ is not well-defined. We overcome this complication by introducing an additional ''damping factor'' $(1+|f(z)|^4)^{-1} |f^{'}(z)|^{2}$($ \geqslant 0$)
and considering a system
$\overline{\mathcal{N} }(f)$  of the form
\begin{equation}
\label{vgl3}
\frac{dz}{dt}= -(1+|f(z)|^4)^{-1}  \overline{f^{'} (z)}  f(z).
\end{equation}
Clearly, $\overline{\mathcal{N} }(f)$  may be regarded as another Ôinfinitesimal versionÕ of Newton's iteration (\ref{vgl1}). 
Note that, where both $\mathcal{N}(f)$ and $\overline{\mathcal{N} }(f)$  are well-defined, their phase portraits coincide, including the orientations of the trajectories (cf. Fig. \ref{Figure1}). Moreover, $\overline{\mathcal{N} }(f)$  is a smooth, even real (but not complex) analytic vector field on the {\it whole} plane. For our aims it is enough that $\overline{\mathcal{N} }(f)$ is of class $C^{1}$.
We refer to $\overline{\mathcal{N} }(f)$  as to a {\it desingularized Newton flow} for $f$ on $\mathbb{C}$.

\begin{figure} [htbp]
\begin{center}
\includegraphics*[height=4cm, width=5cm]{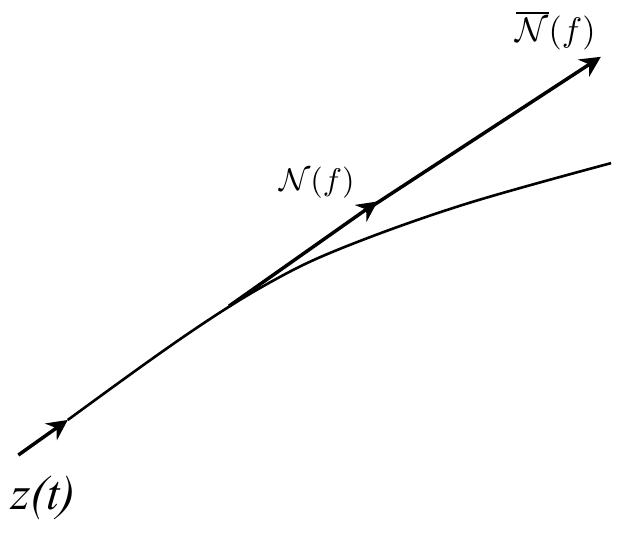}
\caption {$\mathcal{N} (f)$ versus $ \overline{\mathcal{N}} (f)$.} 
\label{Figure1}
\end{center}
\end{figure}

\noindent
Integration of (\ref{vgl2})
yields:
\begin{equation}
  \label{vgl4}
  f(z(t)) = e^{-t} f(z_0)  ,\;  z(0)=z_{0},
\end{equation}
where $z(t)$ denotes the maximal trajectory for $\mathcal{N}(f)$ through a point $z_{0}$. So we have
\begin{equation}
\label{vgl5}
\text{$\mathcal{N}(f)$-trajectories and also those of $\overline{\mathcal{N} }(f)$, are made up of lines arg $f(z) =\text{constant. }$   }
\end{equation}
It is easily verified that these Newton flows fulfil a {\it duality property} which will play an important role in the sequel:
\begin{equation}
\label{vgl6}
  \mathcal{N} (f) = - \mathcal{N} (\dfrac{1}{f}) \text{ and }\overline{\mathcal{N} }(f)= - \overline{\mathcal{N}} (\dfrac{1}{f}).
\end{equation}

\noindent
As a consequence of (\ref{vgl5}), (\ref{vgl6}), and using properties of (multi-)conformal mappings, we picture the local phase portraits of $\mathcal{N}(f)$ and $\overline{\mathcal{N} }(f)$ around their equilibria. See the comment on Fig. \ref{Figure2}, where $N(f)$, $P(f)$ and $C(f)$ stands for respectively the set of zeros, poles and critical points of $f$.

\begin{figure} [htbp]
\begin{center}
\includegraphics*[height=4cm, width=14cm]{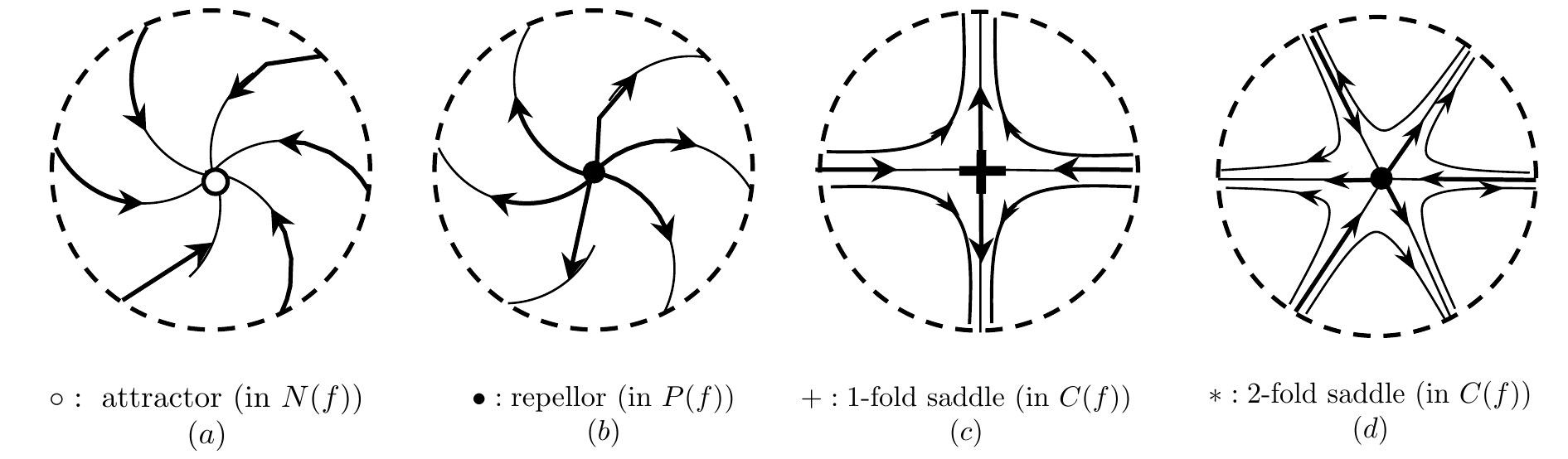}
\caption {Local phase portraits around equilibria of $\overline{\mathcal{N} }(f)$ }
\label{Figure2}
\end{center}
\end{figure}

\noindent
\underline{Comment on Fig. \ref{Figure2}}:

\noindent
Fig. \ref{Figure2}-$(a),(b)$:
In case of a $k$-fold zero (pole) the Newton flow exhibits an attractor (repellor) and  each (principal) value of arg$f$ appears precisely $k$ times on equally distributed incoming (outgoing) trajectories. 
Moreover, as for the (positively measured) angle between two different incoming (outgoing) trajectories, they intersect under a non vanishing angle $\!\frac{\Delta}{k}$, where $\Delta$ stands for the difference of the arg$f$ values on these trajectories, i.e., these equilibria are star nodes. 
In the sequel we will use: If two incoming (outgoing) trajectories at a {\it simple} zero (pole) admit the same arg$f$ value, those trajectories coincide.

\noindent
Fig. \ref{Figure2}-$(c),(d)$: 
In case of a $k$-fold critical point (i.e., a $k$-fold zero for $f'$, no zero for $f$) the Newton flow exhibits a $k$-fold saddle, the stable (unstable) separatrices being equally distributed around this point. The two unstable (stable) separatrices at a 1-fold saddle, see Fig. \ref{Figure2}(c), constitute the ÔlocalÕ unstable (stable) manifold at this saddle point.\\

\begin{figure}
\centering
\includegraphics[width=2.9in]{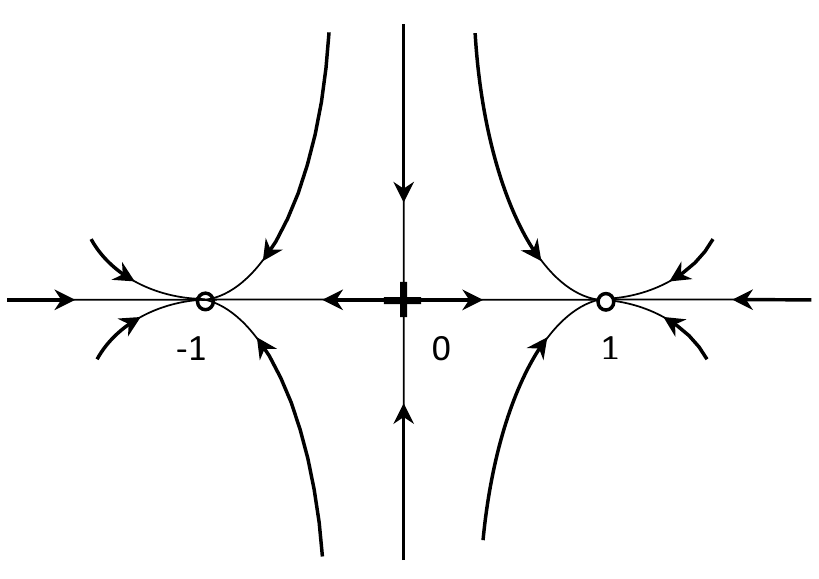}
\caption{ Phaseportrait $\overline{\mathcal{N}}(z^{2}-1)$}
\label{Figure3}
\end{figure} 

\begin{figure}
\centering
\includegraphics[width=2.7in]{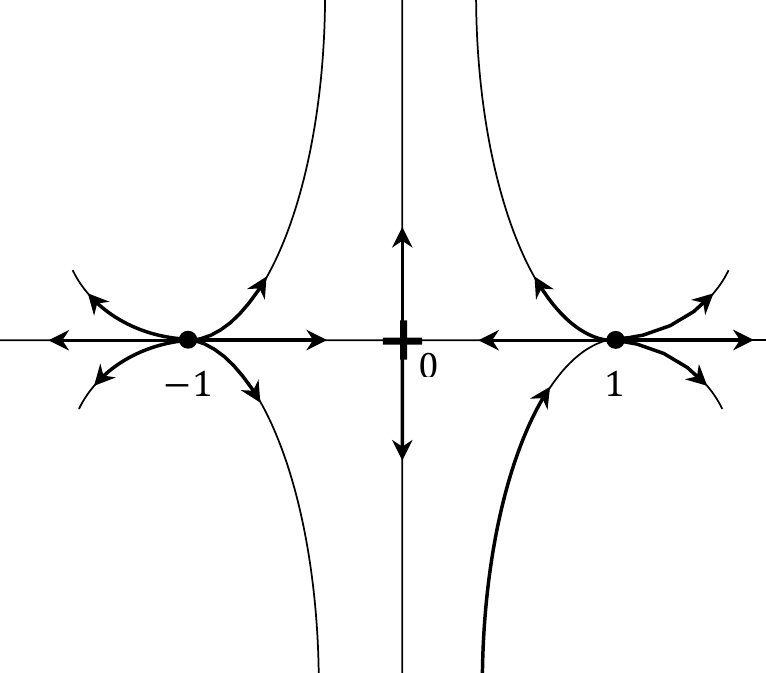}
\caption{ Phaseportrait $\overline{\mathcal{N}}(\frac{1}{z^{2}-1})$}
\label{Figure4}
\end{figure}

In the sequel we shall need:
\begin{remark}
\label{R1.1}
Let $z_{0}$ be either a {\it simple} zero, pole or critical point for $f$. Then $z_{0}$ is a hyperbolic\footnote{An equilibrium for a $C^{1}$- vector field on $\mathbb{R}^{2}$ is called hyperbolic if the Jacobi matrix at this equilibrium has only eigenvalues with non vanishing real parts (cf.\cite{H/S2}).} equilibrium for $\overline{\mathcal{N} }(f)$. (In case of a zero or critical point for $f$, this follows by inspection of the linearization of the r.h.s. of $\overline{\mathcal{N} }(f)$: in case of a pole use (\ref{vgl6}).)
\end{remark}
\begin{remark}
\label{R1.2}
(Desingularized meromorphic Newton flows in $\mathbb{R}^{2}$-setting)
\newline
\noindent
If we put $F: ({\rm Re}(z), {\rm Im}(z))^{T}(=(x_{1}, x_{2})^{T}) \mapsto ({\rm Re}f(z), {\rm Im}f(z))^{T}$, the desingularized Newton flow $\overline{\mathcal{N} }(f)$ takes the form
\begin{align}
\label{vgl7}
\frac{d}{dt}(x_{1}, x_{2})^{T}&
=-[1+|F(x_{1}, x_{2})|^{4}]^{-1} \det (DF(x_{1}, x_{2})) (DF)^{-1}(x_{1}, x_{2})F(x_{1}, x_{2}) \\ \notag
&=-[1+|F(x_{1}, x_{2})|^{4}]^{-1} \tilde{D}F(x_{1}, x_{2}) F(x_{1}, x_{2}),
\end{align}
where $(\cdot)^{T}$ stands for transpose, and $\tilde{D}F(\cdot)$ for the co-factor (adjoint) matrix\footnote{i.e. $\tilde{D}F(x_{1}, x_{2}) \cdot
DF(x_{1}, x_{2})=\det (DF(x_{1}, x_{2}))I_{2}$, where $I_{2}$ stands for the $2 \times 2$-unit matrix.} of the Jacobi matrix $DF(\cdot)$ of $F$. (The r.h.s. of (\ref{vgl7}) vanishes at points corresponding to poles of $f$)
\end{remark}

We end up with some pictures illustrating the above explanation.

\begin{figure}[h!]
\centering
\includegraphics[width=3.1in]{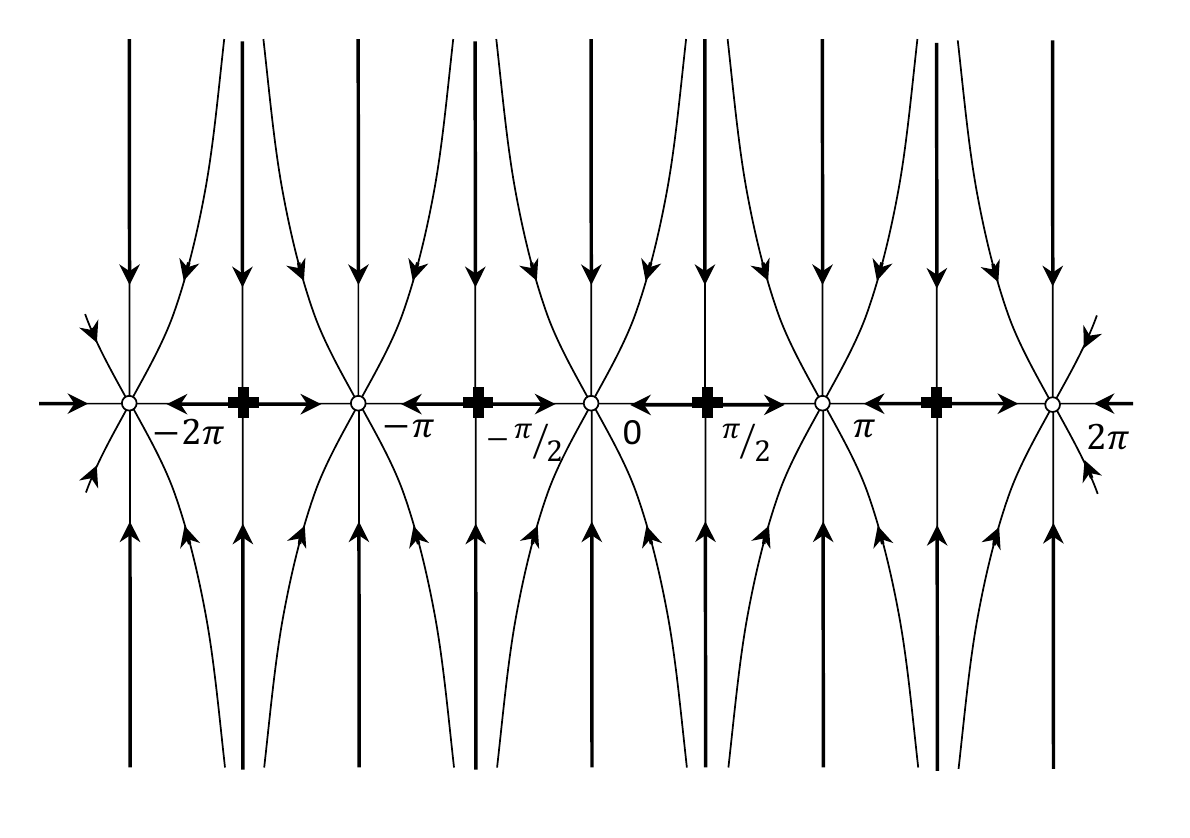}
\caption{ Phaseportrait $\overline{\mathcal{N}}(\sin z)$}
\label{Figure5}
\end{figure} 

\newpage
\begin{figure}[h!]
\centering
\includegraphics[width=3in]{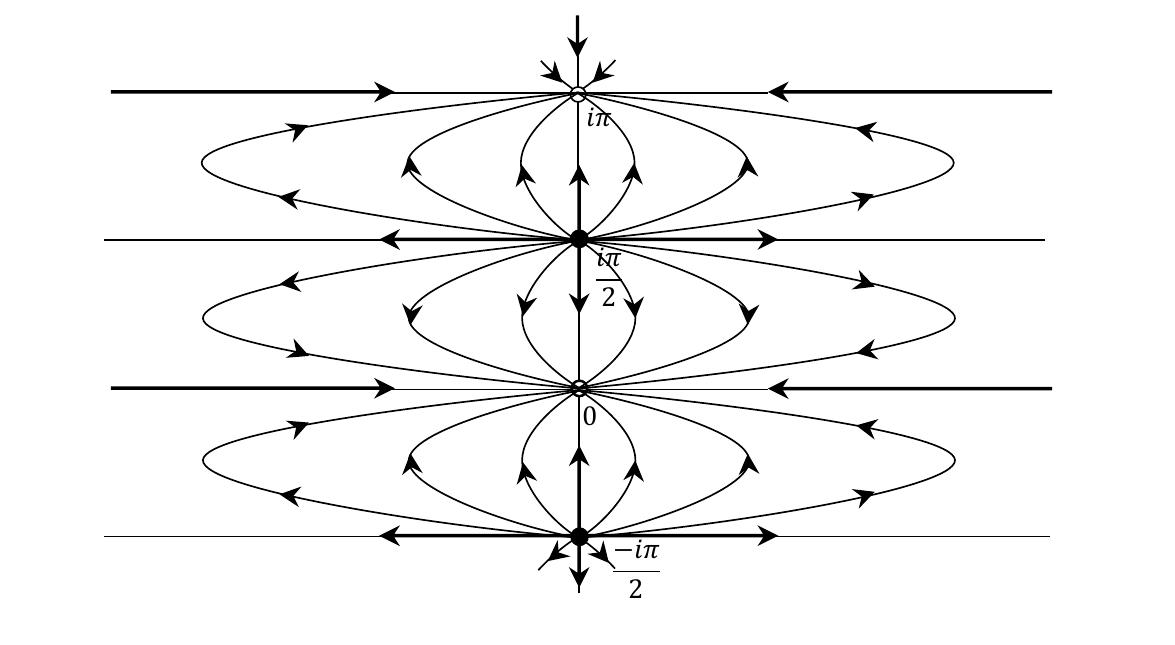}
\caption{ Phaseportrait of $\overline{\mathcal{N}}({\rm tanh} z)$}
\label{Figure6}
\end{figure}

\section{Rational Newton flows; the purpose of the paper}
\label{sec2}

Here we present some earlier results on meromorphic Newton flows in the special case of rational functions. 
Throughout this section, let $f$ be a ({\it non-constant}) rational function. By means of the transformation $w\!=\!\frac{1}{z}$ we may regard $f$ as a function on the extended complex plane 
$\mathbb{C} \cup \{ z=\infty\}$. As usual, we identify the latter set with the sphere $S^{2}$ (
as a Riemann surface) and the set $\mathcal{R}$ of extended functions $f$ with the set of all (non-constant) meromorphic functions on $S^{2}$.
The transformation $w\!\!=\!\!\frac{1}{z}$ turns the Ô{\it planar rational Newton flow}Õ $\overline{\mathcal{N} }(f)$, $f \in \mathcal{R}$, into a {\it smooth} vector field on $S^{2}$, denoted  $\overline{\overline{\mathcal{N}} }(f)$, cf. \cite{JJT3}, \cite{Twi}.
In the theory on such vector fields the concept of structural stability plays an important role, see e.g. \cite{Peix1} or \cite{H/S2}.
Roughly speaking, structural stability of $\overline{\overline{\mathcal{N}} }(f)$ means ''topological invariance of its phase portrait under sufficiently small perturbations of the problem data''. Here we briefly summarize the results as obtained by Jongen et al., Shub (cf. \cite{JJT2}, \cite{JJT3}, \cite{JJT4}, \cite{Shub}):
\begin{theorem} $($Structural stability for rational Newton flows$)$\\
\label{thmS2} 
Let $f \in \mathcal{R}$, then:
\begin{enumerate}
\item[(i)] Characterization:\\	
The flow $\overline{\overline{\mathcal{N}} }(f)$, $f \in \mathcal{R}$,  is structurally stable iff
$f$ fulfils the following conditions:
\begin{itemize}
\item All finite zeros and poles for $f$ are simple. 
\item  All critical points for $f$, possibly including $z=\infty$, are simple.
\item  No two critical points for $f$ are connected by an $\overline{\overline{\mathcal{N}} }(f)$-trajectory.
\end{itemize}
\item[(ii)] Genericity:\\
For ``almost all'' functions $f$ in $\mathcal{R}$, the flows $\overline{\overline{\mathcal{N}} }(f)$ are structurally stable, i.e. the functions $f$ 
as in $($i$)$ constitute an open and dense subset of $\mathcal{R}$ $($w.r.t. an appropriate topology on $\mathcal{R})$. 
\item[(iii)] Classification:\\
 The conjugacy classes of the stucturally stable flows $\overline{\overline{\mathcal{N}} }(f)$ can be classified in terms of 
 certain sphere graphs that are generated by the phase portraits of these flows.
\item[(iv)] Representation:\\
Up to conjugacy for flows and $($topological$)$ equivalency for graphs, there is a 
1-1- correspondence between the set of all structurally stable flows $\overline{\overline{\mathcal{N}} }(f)$ and the set of
 all so-called Newton graphs, i.e., cellularly embedded sphere graphs that fulfil some
 combinatorial $($Hall $)$ condition. 
\end{enumerate}
\end{theorem}
\noindent
The purpose of the present paper and its sequel \cite{HT2} is to find out wether similar results hold for {\it elliptic} Newton flows (i.e., meromorphic Newton flows in the case of elliptic functions), be it that, especially in the cases (iii) Classification and
(iv) Representation, the proofs are much harder, see \cite{HT2}. In the present paper we focus on the first two properties mentioned in Theorem \ref{thmS2}: characterization and genericity. 

Phase portraits of rational Newton flows (even structurally stable) on $\mathbb{C}$ are presented in Fig. \ref{Figure3} and \ref{Figure4}. The simplest example of a spherical rational Newton flow is the so-called North-South flow, given by $\overline{\overline{\mathcal{N}} }(z^{n}), n \geqslant 1$, see Fig. \ref{Figure7}; structurally stable if
$n=1$. Intuitively, it is clear
that the phase portraits of $\overline{\overline{\mathcal{N}} }(z^{n})$ and $\overline{\overline{\mathcal{N}} }((\frac{z-a}{z-b})^{n}), a\neq b, $ are topologically equivalent (i.e., equal up to conjugacy), see Fig.\ref{Figure7} and \ref{Figure8}.

One of the first applications of Newton flows was Branin's method for solving non linear problems, 
see \cite{Bra}, \cite{Gom} and \cite{JJT2}. It was Smale, see \cite{Smale85}, who stressed the importance for complexity theory of classifying Newton graphs on the sphere
that determine the desingularized rational Newton flows.
This was done for a class of polynomials in \cite{Shub} and in general in \cite{JJT4}. Also in the elliptic case, the classification of so-called elliptic Newton graphs on the torus,
which determine the desingularized elliptic Newton flows, has implications for complexity theory, see \cite{HT2}.

\begin{figure}[h!]
\centering
\includegraphics [width=3.4in]{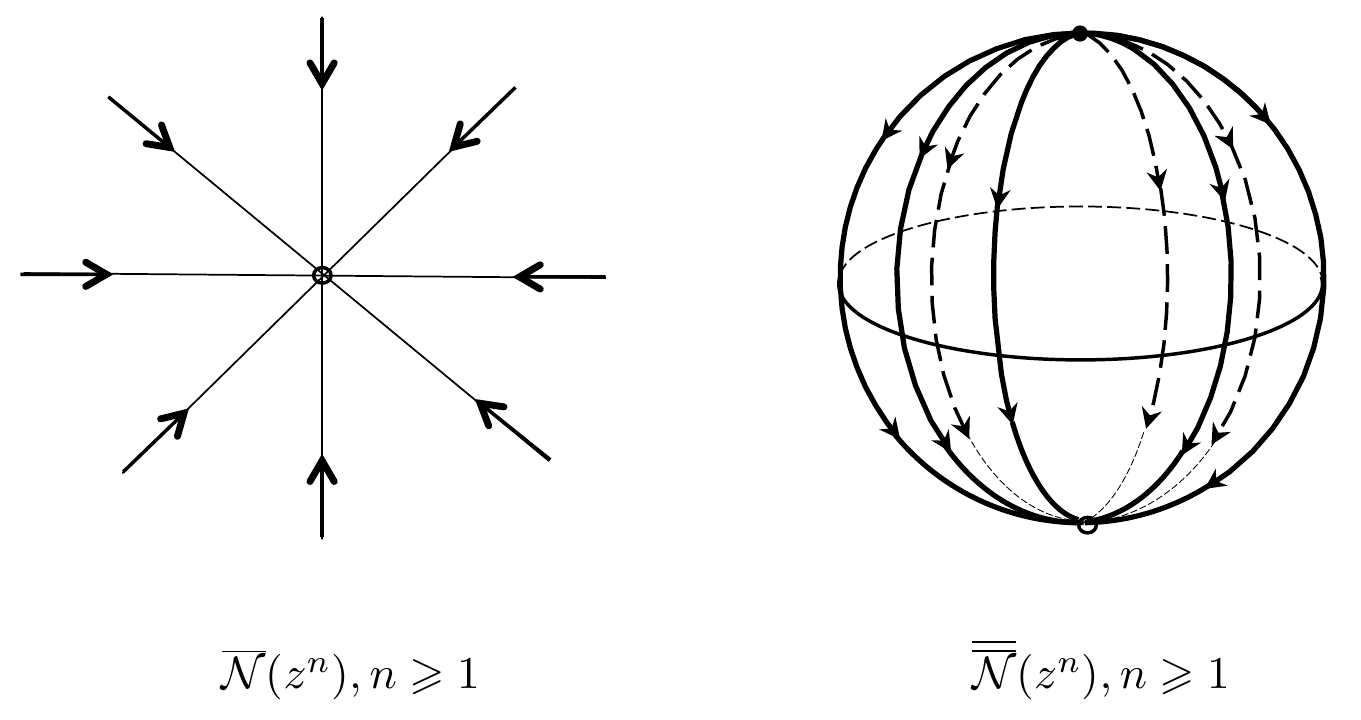}
\caption{ The planar and spherical North-South flow}
\label{Figure7}
\end{figure} 
\begin{figure}[h!]
\centering
\includegraphics[width=1.6in]{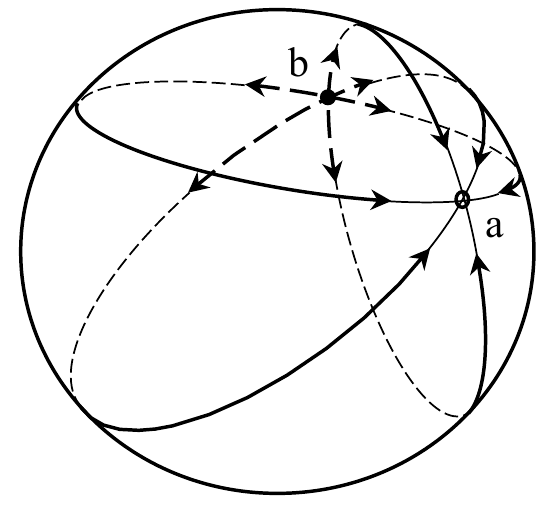}
\caption{ $\overline{\overline{\mathcal{N}}}((\frac{z-a}{z-b})^{n})$} 
\label{Figure8}
\end{figure}

\section{Elliptic Newton flows: definition}
\label{sec3}
Throughout this section, let $f$ be a (non-constant) elliptic, i.e.\;a meromorphic and doubly periodic function of order $r$ $(2 \leqslant r < \infty)$ with $(\omega_{1},\; \omega_{2})$ as a pair of {\it basic periods.}\footnote{i.e. each period is of the form $n \omega_{1}+m\omega_{2}, n, m \in \mathbb{Z}.$ In particular, ${\rm Im}\frac{\omega_{2}}{\omega_{1}} \neq 0$ (cf. \cite{M2}, \cite{Cha}).} We always assume that ${\rm Im}\frac{\omega_{2}}{\omega_{1}} >0.$ The associated period lattice is denoted by $\Lambda$, and $P_{\omega_{1}, \; \omega_{2}}$ stands for the ''half open/half closed'' period parallelogram $\{t_{1}\omega_{1} + t_{2}\omega_{2} \mid 0 \leqslant t_{1}<1, \;0 \leqslant t_{2}<1\}$.
On $P_{\omega_{1},\omega_{2}}$, the function $f $ has $r$ zeros and $r $ poles (counted by multiplicity).

By Liouville's Theorem, these sets of zeros and poles determine $f$ up to a multiplicative constant $C(\neq 0)$,
and thus also the class $[f]$ of all elliptic functions of the form $Cf , C(\neq 0).$

Let $T_{\omega_{1},\omega_{2}}$ be the torus obtained from $P_{\omega_{1},\omega_{2}}$ by identifying opposite sides in the boundary of this parallelogram. The planar, desingularized Newton flow $\overline{\mathcal{N} }(f)$ is doubly periodic on $\mathbb{C}$ with periods $(\omega_{1}, \omega_{2})$. Hence, this flow may be interpreted 
as a $C^{1}$-(even smooth, but nowhere meromorphic) vector field, say 
$\overline{\overline{\mathcal{N}}}_{\omega_{1},\omega_{2}} (f)$,
on $T_{\omega_{1},\omega_{2}}$; its trajectories correspond to the lines
${\rm arg} f(z)=$ constant, cf. (\ref{vgl5}).  
We refer to 
$\overline{\overline{\mathcal{N}}}_{\omega_{1},\omega_{2}} (f)$
as the {\it (desingularized) elliptic Newton flow} for $f$ on $T_{\omega_{1},\omega_{2}}$.

If $g$ is another function in $[f]$, the planar flows $\overline{\mathcal{N} }(g)$ and $\overline{\mathcal{N} }(f)$ have equal phase portraits,  as follows by inspection of the expressions of these flows in Section 1; see also Fig.\ref{Figure1}. Hence, the flows 
$\overline{\overline{\mathcal{N}}}_{\omega_{1},\omega_{2}} (f)$
and 
$\overline{\overline{\mathcal{N}}}_{\omega_{1},\omega_{2}} (g)$
, both defined on $T_{\omega_{1},\omega_{2}}$, have equal phase portraits.

Next, we choose another pair of basic periods for $f$, say $(\omega'_{1}, \omega'_{2})$, with ${\rm Im}\frac{\omega'_{2}}{\omega'_{1}} >0$, i.e. $(\omega_{1}, \omega_{2})$ and $(\omega'_{1}, \omega'_{2})$ generate the same lattice $\Lambda$ and are related by a unimodular\footnote{$M$ is given by a $2 \times 2$-matrix with coefficients in $\mathbb{Z}$ and determinant +1 (cf. \cite{Cha}).} linear transfomation $M$. 
The function $f$ has on $P_{\omega_{1},\omega_{2}}$ and $P_{\omega^{'}_{1},\omega^{'}_{2}}$ the same zeros/poles (up to congruency mod $\Lambda$).

\noindent
The above introduction of the concept ``elliptic Newton flow for $f $'' leads to different flows (being defined on different tori $T_{\omega_{1},\omega_{2}}$ and $T_{\omega^{'}_{1},\omega^{'}_{2}}$). 
 In order to see the effect of $M$ on the planar flow $\overline{\mathcal{N} }(f)$ we turn over to its $\mathbb{R}^{2}$-setting (cf. Remark 1.2) and apply the chain rule. We put $M(x)=y, x=(x_{1}, x_{2})^{T}, y=(y_{1}, y_{2})^{T}$ in $\mathbb{R}^{2}$ and find:
 \begin{equation}
 \label{verg8}
\frac{dy}{dt}=-[1 +|F(y)|^{4}]^{-1} 
\det(M) \tilde{D}F(y)F(y).
\end{equation}
Hence, because $\det(M)=1$, the phase portrait of $\overline{\mathcal{N} }(f)$ is not changed.
Since $M$ induces a homeomorphism between $T_{\omega_{1},\omega_{2}}$ and $T_{\omega^{'}_{1},\omega^{'}_{2}}$, we find on the level of the toroidal flows:
\begin{lemma}
\label{L3.1}
Let $(\omega_{1}, \omega_{2})$ and $(\omega^{'}_{1}, \omega^{'}_{2})$  be pairs of basic periods for $f$.
Then, the unimodular mapping from $(\omega_{1}, \omega_{2})$ to $(\omega^{'}_{1}, \omega^{'}_{2})$ induces a homeomorphism from $T_{\omega_{1},\omega_{2}}$ to $T_{\omega^{'}_{1},\omega^{'}_{2}}$ such that the phase portraits of 
$\overline{\overline{\mathcal{N}}}_{\omega_{1},\omega_{2}} (f)$
and 
$\overline{\overline{\mathcal{N}}}_{\omega^{'}_{1},\omega^{'}_{2}} (f)$
correspond under this homeomorphism, thereby respecting the orientations of the trajectories.
\end{lemma}
So, from a topological point of view, the flows $\overline{\overline{\mathcal{N}}}_{\omega_{1},\omega_{2}} (f)$
and 
$\overline{\overline{\mathcal{N}}}_{\omega^{'}_{1},\omega^{'}_{2}} (f)$ may be considered as equal. We emphasize that the linear map $M$ does not act on the meromorphic object $\mathcal{N}(f)$, but on the phase portraits of its desingularized toroidal $C^{1}$-versions.
The above lemma, together with the preceding observation, leads to:
\begin{definition}
\label{D3.2}
If $f$ is an elliptic function of order $r$, then: 
\begin{enumerate}
\item 
The elliptic Newton flow for $f$ , denoted $\overline{\overline{\mathcal{N}}} ([f])$, is the collection of all flows 
$\overline{\overline{\mathcal{N}}}_{\omega_{1},\omega_{2}} (g)$
, for any $g \in [f]$ and any pair $(\omega_{1}, \omega_{2})$ generating the  period lattice $\Lambda$ of $f$.
\item
The set of all elliptic Newton flows of order $r$ with respect to a given period lattice $\Lambda$ is denoted $N_{r}(\Lambda)$.
\end{enumerate}
\end{definition}

\noindent
This definition 
might look rather complicated. However, a natural interpretation is possible. To see this, let us consider the quotient space $T(\Lambda)\!\!=\!\!\mathbb{C}/ \Lambda$, endowed with the complex analytic structure\footnote{As coordinate neighborhoods in $T(\Lambda)$ take open subsets of $\mathbb{C}$ that contain no points congruent to one another mod $\Lambda$.}  determined by the pair $(\omega_{1}, \omega_{2})$. A pair $(\omega^{'}_{1}, \omega^{'}_{2})$, related to $(\omega_{1}, \omega_{2})$ by a unimodular map, generates the same lattice $\Lambda$ and determines on $T(\Lambda)$ another -but isomorphic- complex analytic structure (cf. \cite{Gun}). Each parallelogram 
$P_{\omega_{1},\omega_{2}}$ resp. $P_{\omega^{'}_{1},\omega^{'}_{2}}$  contains precisely one representative for each of the classes $\mathrm {mod} \, \Lambda$. Hence, the tori $T_{\omega_{1},\omega_{2}}$ and $T_{\omega^{'}_{1},\omega^{'}_{2}}$ may be identified with $T(\Lambda)$, endowed with 
isomorphic complex analytic 
structures. Now, the  flows 
$\overline{\overline{\mathcal{N}}}_{\omega_{1},\omega_{2}} (f)$ and $\overline{\overline{\mathcal{N}}}_{\omega^{'}_{1},\omega^{'}_{2}} (f)$
can be interpreted
as  smooth flows on $T(\Lambda)$ with {\it the same phase portraits}.
Regarding flows on $T(\Lambda)$ with the same phase portraits as equal, compare the ``desingularization'' step leading from (\ref{vgl2}) to  (\ref{vgl3}) and see also Fig. \ref{Figure1}, we may interprete the elliptic Newton flow $\overline{\overline{\mathcal{N}}} ([f])$ as a smooth vector field on the compact torus $T(\Lambda)$.
Consequently, it is allowed to apply the theory for $C^{1}$-vector fields on compact two-dimensional differential manifolds. For example: Since 
there are no closed orbits by (\ref{vgl4}), and applying the Poincar\'e-Bendixson-Schwartz Theorem, cf.  \cite{JJT1}, \cite{JJT2}, \cite{Peix1}, we find: 
\begin{lemma}
\label{L3.3}
The limiting set of any (maximal) trajectory of $\overline{\overline{\mathcal{N}}} ([f])$ tends -for increasing $t$- to either a zero or a critical point for $f$ on $T(\Lambda)$, and -for decreasing $t$- to either a pole or a critical point for $f$ on 
$T(\Lambda)$. 
\end{lemma}
We also have:
\begin{remark}
\label{R3.4}
Hyperbolic equilibria for $\overline{\mathcal{N}} (f)$ correspond to such equilibria for $\overline{\overline{\mathcal{N}}} ([f])$.
\end{remark}

\section{Elliptic Newton flows: representation}
\label{sec4}

Let $f$ be 
as 
in Section \ref{sec3}, i.e., an elliptic function of order $r$ $(2 \leqslant r < \infty)$ with $(\omega_{1}, \omega_{2})$, ${\rm Im}\frac{\omega_{2}}{\omega_{1}} >0$, an (arbitrary) pair of basic periods generating a period lattice $\Lambda$.  The set of all such functions is denoted by $E_{r}(\Lambda)$.

Let 
the zeros and  poles for $f$ on $P_{\omega_{1},\omega_{2}}$ be $a_{1}, \! \cdots \!\!,a_{r} $, resp. $b_{1}, \! \cdots \! ,b_{r} $ (counted by multiplicity). Then we have: (cf. \cite{M2})
\begin{equation}
\label{vgl9}
a_{i} \neq b_{j},i, j=1, \! \cdots \! \!,r \text{ and }a_{1}+  \cdots +a_{r} =b_{1} + \cdots  +b_{r} \text{ mod }\Lambda.
\end{equation}
We may consider $f$ as a meromorphic function on the  quotient space $T(\Lambda)\!\!=\!\!\mathbb{C}/ \Lambda$.
The zeros and poles for $f$ on $T(\Lambda)$ are given by respectively: 
$[a_{1}], \! \cdots \!, [a_{r}] $ and  $[b_{1}], \! \cdots \!, [b_{r}] $, where $[\cdot]$ stands for the congruency class $\mathrm {mod} \, \Lambda$
of a number in $\mathbb{C}$. Apparently, from (\ref{vgl9}) it follows:
\begin{equation}
\label{vgl10}
[a_{i}] \neq [b_{j}],i, j=1, \! \cdots \! ,r \text{ and }[a_{1}]+ \! \cdots \! +[a_{r}] =[b_{1}] +\! \cdots \! +[b_{r}].
\end{equation}
Moreover, a parallelogram of the type $P_{\omega_{1},\omega_{2}}$ contains one representative of each of these classes: the $r$ zeros/poles for $f$ on this parallelogram.

An elliptic Newton flow $\overline{\overline{\mathcal{N}}} ([f])(\in N_{r}(\Lambda))$ corresponds uniquely to the class $[f]$. So we may identify the set $N_{r}(\Lambda)$ with the set  $\{[f ] | f \in E_{r}(\Lambda)\}$.

On its turn, the class $[f]$ is uniquely determined (cf. Section \ref{sec3}) by sets of zeros/poles, say $\{ a_{1}, \! \cdots \! ,a_{r} \}$/ $\{ b_{1}, \! \cdots \! ,b_{r} \}$ 
, both situated in some $P_{\omega_{1},\omega_{2}}$. Thus, the  sets 
$$
\{[a_{1}], \! \cdots \!, [a_{r}] \},  \{[b_{1}], \! \cdots \! , [b_{r}]\} 
$$ 
fulfill the conditions (\ref{vgl10}). Conversely, we have:
\newline
Let two sets $\{[a_{1}], \! \cdots \!, [a_{r}] \}, \; \{[b_{1}], \! \cdots \!
, [b_{r}]\}$ of classes 
$\mathrm {mod} \, \Lambda$
(repetitions permitted!), fulfilling conditions (\ref{vgl10}), be given.  
Assume that the representatives, $a_{1}, \! \cdots \! ,a_{r}$ and $b_{1}, \! \cdots \! ,b_{r}$, of these classes are situated in a half open/half closed parallelogram spanned by an (arbitrary) pair of basic periods of $\Lambda$. We put 
$$
b^{'}_{r} = a_{1}+ \! \cdots \! +a_{r} - b_{1}- \! \cdots \! -b_{r-1}, \text{ thus }[b^{'}_{r}]=[b_{r}].
$$
Consider functions of the form
\begin{equation}
\label{vgl11}
C \frac{\sigma(z - a_1)\! \cdots \! \sigma (z - a_r)} {\sigma(z - b_1) \! \cdots \! \sigma(z - b_{r-1})  \sigma (z - b^{'}_r)},
\end{equation}
where $C(\neq 0)$ is an arbitrary constant, and $\sigma$ stands for 
Weierstrass' sigma function corresponding to $\Lambda$ (cf. \cite{M2}).
Since $\sigma$ is a holomorphic, quasi-periodic
function with only simple zeros
at the lattice points of $\Lambda$, a function given by (\ref{vgl11}) is elliptic 
of order $r$; the zeros and poles are $ a_{1}, \! \cdots \! ,a_{r} $ resp. $ b_{1}, \! \cdots \! ,b_{r} $ (cf. \cite{M2}).  Such a function determines precisely one element of $N_{r}(\Lambda)$.  (Note that if we choose $ a_{1}, \! \cdots \! ,a_{r} ,b_{1} \! \cdots \! b_{r}$ 
in any other period parallelogram, we obtain a representative of the same Newton flow, cf. Section \ref{sec3}; moreover, the incidental role of $b_{r}$ does not affect generality)

Altogether, we have proved:
\begin{lemma}
\label{L4.1}
The flows in $N_{r}(\Lambda)$ are given
by 
all ordered pairs $(\{ [a_{1}],
\! \cdots \! ,[a_{r} ]\}$, $\{ [b_{1}], \! \cdots \! ,[b_{r} ]\})$ of sets of classes $\mathrm {mod} \, \Lambda$ that fulfil 
(\ref{vgl10}).
\end{lemma}
\begin{remark}
\label{R4.2}
Interchanging the roles of $(\{ [a_{1}],
\! \cdots \! ,[a_{r} ]\}$
and $\{ [b_{1}], \! \cdots \! ,[b_{r} ]\})$
reflects the duality property, cf. (\ref{vgl6}). In fact, we have $\overline{\overline{\mathcal{N}}} ([\frac{1}{f}]) =- \overline{\overline{\mathcal{N}}} ([f])$.
\end{remark}
\noindent
On the subset 
$V_{r}(\Lambda)$ in $T^{r}(\Lambda) \times T^{r}(\Lambda)$ of pairs $(c,d), \;c=([c_{1}], \! \cdots \! , [c_{r}]), d=([d_{1}], \! \cdots \! , [d_{r}])$, that fulfil condition (\ref{vgl10}), we define an equivalence relation $(\approx)$:
$$
(c,d) \approx (c^{'},d^{'}) \text{ iff } \{ [c_{1}], \! \cdots \! , [c_{r}] \}=\{ [c^{'}_{1}], \! \cdots \! , [c^{'}_{r}] \} \text{ and }\{ [d_{1}], \! \cdots \! , [d_{r}] \}=\{ [d^{'}_{1}], \! \cdots \! , [d^{'}_{r}] \}
$$
\noindent
$\underline{\text{The topology $\tau_{0}$ on $E_{r}(\Lambda)$}}$\\

\noindent
Clearly, the set  $V_{r}(\Lambda) /\! \! \approx$
may be identified with a representation space for $N_{r}(\Lambda)$ 
and thus for $\{[f ] \mid f \in E_{r}(\Lambda)\}$. 
This space can be endowed with a topology which is successively induced by the quotient topology on $T(\Lambda)=(\mathbb{C}/\Lambda))$, the product topology on $T^{r}(\Lambda) \times T^{r}(\Lambda)$, the relative topology on $V_{r}(\Lambda)$ as a subset of $T^{r}(\Lambda) \times T^{r}(\Lambda)$, and the quotient topology w.r.t. the relation$\; \approx$.\\
Finally, we endow $E_{r}(\Lambda)$ with the weakest topology, say $\tau_{0}$, making the mapping 
$$
E_{r}(\Lambda) \rightarrow N_{r}(\Lambda)
: f \mapsto [f]
$$
continuous.

The topology $\tau_{0}$ on $E_{r}(\Lambda)$ is induced by the Euclidean topology on $\mathbb{C}$, and is natural in the following sense: Given $f $ in $ E_{r}(\Lambda)$ and $\epsilon >0$ sufficiently small, a $\tau_{0}$-neighbourhood $\mathcal{O}$ of $f$ exists such that for any $g \in \mathcal{O}$ , the zeros (poles) for $g$ are contained in $\varepsilon$-neighbourhoods of the zeros (poles) for $f$.

Uptill now, we dealt with elliptic Newton flows  $\overline{\overline{\mathcal{N}}} (f)$ with respect to an arbitrary, but fixed, lattice, namely the lattice $\Lambda$ for $f$. Now, we turn over to a different 
lattice, say $\Lambda^{*}$, i.e., pairs of basic periods for $\Lambda$ and $\Lambda^{*}$ are not necessarily related by a unimodular transformation.
Firstly, we treat a simple case: For $\alpha \in \mathbb{C}\backslash \{0\}$
, we define $f^{\alpha}(z)=f(\alpha^{-1}z)$. Thus $f^{\alpha}$ is an elliptic function, of order $r$, with basic periods $(\alpha \omega_{1}, \alpha \omega_{2})$ generating  the lattice $\Lambda^{*}=\alpha \Lambda$.

The following lemma is in the same spirit as Lemma \ref{L3.1}. 
\begin{lemma}
\label{L4.3}
The transformation $z \mapsto w=\alpha z$
induces a homeomorphism from the torus $T_{\omega_{1},\omega_{2}}$ onto $T_{\alpha \omega_{1}, \alpha \omega_{2}}$ , such that the phase portraits of the flows $\overline{\overline{\mathcal{N}}} ([f])
$ and   
$\overline{\overline{\mathcal{N}}} ([f^{\alpha}])
$ correspond  under this homeomorphism, thereby respecting the orientations of the trajectories. 
\end{lemma}
\begin{proof}
Under the transformation $w=\alpha z$, the flow $\overline{\mathcal{N} }(f)$ given by (3), transforms into:
\begin{equation*}
\frac{dw}{dt}= -|\alpha |^{2}(1+|f^{\alpha}(w)|^4)^{-1}  \overline{(f^{\alpha})^{'} (w)}  f^{\alpha}(w).
\end{equation*}
Since $|\alpha |^{2}>0$, this $C^{1}$-flow on $\mathbb{C}$ has the same phase portraits as $\overline{\mathcal{N} }(f^{\alpha})$. 
The assertion follows because the transformation $w=\alpha z$ induces a homeomorphism between the tori 
$T_{\omega_{1},\omega_{2}}$ and $T_{\alpha \omega_{1}, \alpha \omega_{2}}$
\end{proof}
In other words: from a topological point of view, the Newton flows $\overline{\overline{\mathcal{N}}} ([f])\in N_{r}(\Lambda)$ and 
$\overline{\overline{\mathcal{N}}} ([f^{\alpha}])\in N_{r}(\alpha \Lambda)$ may be considered as equal.

More general, we call the Newton flows $\overline{\overline{\mathcal{N}}} ([f])\in N_{r}(\Lambda)$ and $\overline{\overline{\mathcal{N}}} ([g])\in N_{r}(\Lambda^{*})$ equivalent ($\sim$) if they attain representatives, say $\overline{\overline{\mathcal{N}}}_{\omega_{1},\omega_{2}} (f)$, respectively  $\overline{\overline{\mathcal{N}}}_{\omega^{*}_{1},\omega^{*}_{2}} (g)$, and  there is a homeomorphism $T_{\omega_{1},\omega_{2}} \rightarrow T_{\omega^{*}_{1},\omega^{*}_{2}}$ , induced by the linear (over $\mathbb{R}
$) basis transformation 
$(\omega_{1},\omega_{2}) \mapsto
(\omega^{*}_{1},\omega^{*}_{2})$, such that their phase portraits correspond under this homeomorphism, thereby respecting the orientations of the trajectories.
\\ \\
From now on, we choose for $(\omega_{1},\omega_{2})$ a pair of so-called 
{\it reduced}\footnote{The pair of basic periods $(\omega_{1},\omega_{2})$ for $f$  is called {\it reduced} or {\it primitive} if $|\omega_{1}|$ is minimal among all periods for $f$ , whereas $| \omega_{2} |$ is minimal among all periods $\omega$ for $f$  with the property ${\rm Im}\frac{\omega}{\omega_{1}} >0$ (cf. \cite{Cha}).}  periods for $f$, so that the quotient $\tau= \frac{\omega_{2}}{\omega_{1}}$
satisfies the conditions: 
\begin{equation}
\label{vgl12}
\begin{cases}
& \text{Im } \tau  > 0,        | \tau | \geqslant 1,       - \frac{1}{2} \leqslant  \text{Re }  \tau  <\frac{1}{2},\\
&  \text{Re } \tau \leqslant 0, \text{ if }    | \tau |=1
\end{cases}
\end{equation}
(Such a choice is always possible (cf. \!\!\! \cite{Cha}).                                                                                                                    
Moreover, $\tau$ is unique in the following sense:  if $(\omega^{'}_{1},\omega^{'}_{2})$ is another pair of reduced periods for $f$, such that $\tau^{'}\!=\! \frac{\omega^{'}_{2}}{\omega^{'}_{1}}$ also satisfies  the conditions (\ref{vgl12}), then $\tau=\tau^{'}$). 

Together with Lemma \ref{L4.3} this yields:
$$
\text{$\overline{\mathcal{N} }([f])$ is equivalent with 
$\overline{\mathcal{N} }([f^{\frac{1}{\omega_{1}}}])$ in $\Lambda_{1, \tau}$ and $(1, \tau)$ a pair of reduced periods for $f^{\frac{1}{\omega_{1}}}$.}
$$

More generally we have:

\begin{lemma}
\label{L4.4}
Let $f$ be -as before- an elliptic function of order $r$ with $\Lambda$ as period lattice, and let $\Lambda^{*}$ be an arbitrary lattice. Then, there exists a function, say $f^{*}$, with $f^{*} \in E_{r}(\Lambda^{*})$, such that 
$\overline{\overline{\mathcal{N}}} ([f]) \sim \overline{\overline{\mathcal{N}}} ([f^{*}])$.
\end{lemma}
\begin{proof}
Choose $(\omega_{1}, \omega_{2})$ and $(\omega^{*}_{1}, \omega^{*}_{2})$, ${\rm Im}\frac{\omega_{2}^{*}}{\omega_{1}^{*}} >0,$ as basic periods for $\Lambda$ resp. $\Lambda^{*}$ and let $H$ be a linear basis transformation from $(\omega_{1}, \omega_{2})$ to $(\omega^{*}_{1}, \omega^{*}_{2})$. The zeros and poles for $ f$  are represented by the tuples $\{[a_{1}], \!\! \cdots \!\! , [a_{r}] \} $ resp. $ \{[b_{1}], \!\! \cdots \!\! , [b_{r}] \}$ in $P_{\omega_{1},\omega_{2}}$ that fulfil (\ref{vgl9}). Under $H$  these tuples turn into tuples  $\{[a_{1}^{*}], \!\! \cdots \!\! , [a_{r}^{*}] \} $ , $ \{[b_{1}^{*}], \!\! \cdots \!\! , [b_{r}^{*}] \}$ in $P_{\omega_{1}^{*},\omega_{2}^{*}}$, satisfying (\ref{vgl9}) with $\Lambda^{*}$ in the role of $\Lambda$. The latter pair of tuples determines a function $f^{*}$ in $E_{r}(\Lambda^{*})$ and thus, by Lemma \ref{L4.1}, a Newton  flow $\overline{\overline{\mathcal{N}}} ([f^{*}])$ in $N_{r}(\Lambda^{*})$. Now, the chain rule applied to the $\mathbb{R}^{2}$-versions (cf. (\ref{vgl7})) of $\overline{\overline{\mathcal{N}}} ([f])$ and $\overline{\overline{\mathcal{N}}} ([f^{*}])$ yields the assertion.  (Use that, by assumption, $\det(H)>0$).
\end{proof}
\begin{remark}
\label{R4.5}
Note that if 
$\Lambda=\Lambda^{*} $
(thus the basic periods for $\Lambda$, and $\Lambda^{*}$ are related by unimodular transformations), then: $f= f^{*}$.
Also we have: 
$g \in [f] $ implies $g^{*} \in [f^{*}]$.
\end{remark}
We summarize the homogeneity results, specific for continuous elliptic Newton flows, as  obtained in this and the preceding section 
(choose $\Lambda^{*}=\Lambda_{1,\tau}: \tau=i$)
\begin{theorem}$($The canonical form for elliptic Newton flows$)$\\
\label{T4.6}
Given an arbitrary elliptic Newton flow, say $\overline{\overline{\mathcal{N}}} ([f])$, on $T(\Lambda)$, there exists a function $f^{*}$ in $ E_{r}(\Lambda_{1,i})$, of the form
(\ref{vgl11}) with $ a_{1}, \! \cdots \! ,a_{r} $ and $ b_{1}, \! \cdots \! ,b_{r} $ in the parallellogram $P_{1,i}$ and $C=1$, such that
$$
\overline{\overline{\mathcal{N}}} ([f]) \sim \overline{\overline{\mathcal{N}}} ([f^{*}]).
$$
In particular, it is always possible to choose $\Lambda^{*}=\Lambda_{1,\tau}$, $\tau$ as in (\ref{vgl12})
and to apply the linear transformation $(1,\tau) \mapsto (1,i)$.
So that we even may assume that $(1,i)$ is a pair of reduced periods for the corresponding elliptic function on $\Lambda_{1,i}$.\\
\end{theorem}
Hence, in the sequel, we suppress - unless strictly necessary - references to:  $(1, \tau), \Lambda,$ $ \text{ class }[\cdot], \cdots$ and write $\Lambda, T, E_{r}, N(f), \cdots$
instead of $\Lambda(1, \tau), T(\Lambda), E_{r}(\Lambda), N([f]), \cdots$.\\

We end up by presenting two pictures of Newton flows for ${\rm sn}$
, where ${\rm sn}$
stands for a Jacobian function. This is a 2nd order elliptic function, attaining only simple zeros, poles and critical points,
characterized by the basic periods $4K, 2iK'$. 
So does the phase portrait of its Newton flow. Here $K, K'$ are two parameters defined in terms of the Weierstrass function $\wp$.

It turns out that for the phase portrait there are -up to topological equivalency- only two possibilities, corresponding to the form (rectangular or not) of the parallelogram $P_{1, \tau}$ with $(1, \tau) \in D$ and $\tau=(\frac{2iK'}{4K})\text{ mod } 1.$
In fact, the crucial distinction between these two cases is whether there occur so-called ``saddle-connections'' (i.e., (un)stable manifolds connecting saddles) or not (cf. \cite{THBS}).
 Hence, it is sufficient to select for each possibility one suitably chosen example.
See Fig. \ref{Figure9} [non-rectangular, equiharmonic subcase, given by 
$\tau=\frac{1}{3}\sqrt{3} \exp(\frac{\pi i}{3})$]
and Fig. \ref{Figure10} [rectangular subcase given by $\text{Re}\tau=0$]. For a detailed argumentation, see our previous work \cite{THBS}
; compare also the Remarks 2.14 and  2.15 in the forthcoming \cite{HT2}.

Note that in Fig. \ref{Figure9}, \ref{Figure10} the points, labelled by $0, 4K, 2iK'$ and $4K+2iK'$ correspond to the same toroidal zero for ${\rm sn}$
(denoted by $\circ_{1}$)
, whereas both $2K $ and $2K+2iK'$ correspond to the other zero (denoted by $\circ _{2}$).

Similarly, $2K+iK'$ stands for a pole (denoted by $\text{\textbullet}_{3}$) on the torus, the pair $(iK', 4K+iK')$ for the other pole (denoted by $\text{\textbullet}_{4}$). The four torodial critical points (denoted by $+_{5}, \cdots, +_{8}$) are represented by respectively the pairs $(K, K+2iK'), (3K, 3K+2iK')$ and the points $K+iK'$ and $3K+iK'$
; see e.g. \cite{A/S} or \cite{M2}.

It is well-known that the periods $4K, 2iK'$
are not independent of each other, but related via a parameter $m, 0 < m < 1$, see e.g. \cite{A/S}.
In the situation of Fig.\ref{Figure10}: if $m \downarrow 0$, then $4K \rightarrow 2\pi,  \pm 2iK'  \rightarrow  \infty$ and the phase portraits of 
$\overline{\mathcal{N} }({\rm sn})$ 
turn into that of $\overline{\mathcal{N} }({\rm sin})$;
if $m \uparrow 1$, then $ \pm4K \rightarrow  \infty, 2iK' \rightarrow  2 \pi i$ 
and the phase portraits of $\overline{\mathcal{N} }({\rm sn})$ 
turn into that of $\overline{\mathcal{N} }({\rm tanh})$; compare also
 Fig. \ref{Figure5}, \ref{Figure6}.
 
 In part II (cf. \cite{HT2}) of our serial on elliptic Newton flows it has been proved - that up to conjugacy- there is only one 2nd order structurally stable elliptic Newton flow. So that, in a certain sense, the pictures in Fig. \ref{Figure9} represent all examples of possible structurally stable Newton flows of order 2. On the other hand, in the case of order $r=3$, there are several different possibilities as is explained in part III (cf. \cite{HT3}) of our serial.

\begin{figure}
\centering
\includegraphics[width=5.5in]{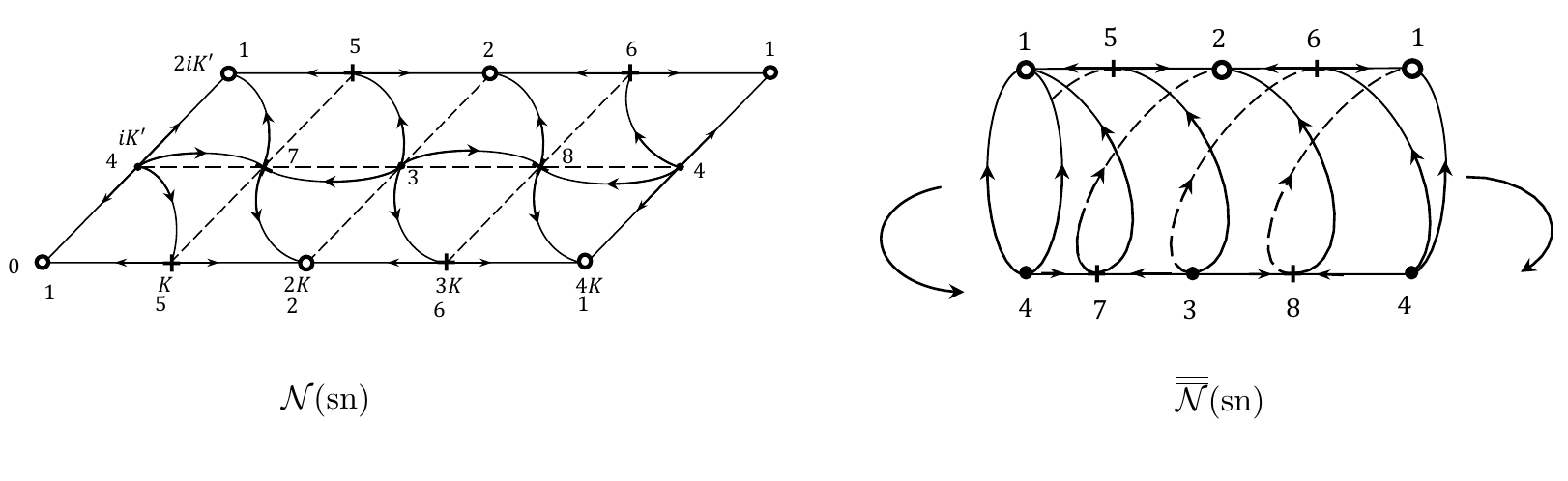}
\caption{ Newton flows for ${\rm sn}$; non-rectangular case; $\tau=\frac{1}{3}\sqrt{3} \exp(\frac{\pi i}{3})$}
\label{Figure9}
\end{figure}

\begin{figure}
\centering
\includegraphics[width=5.5in]{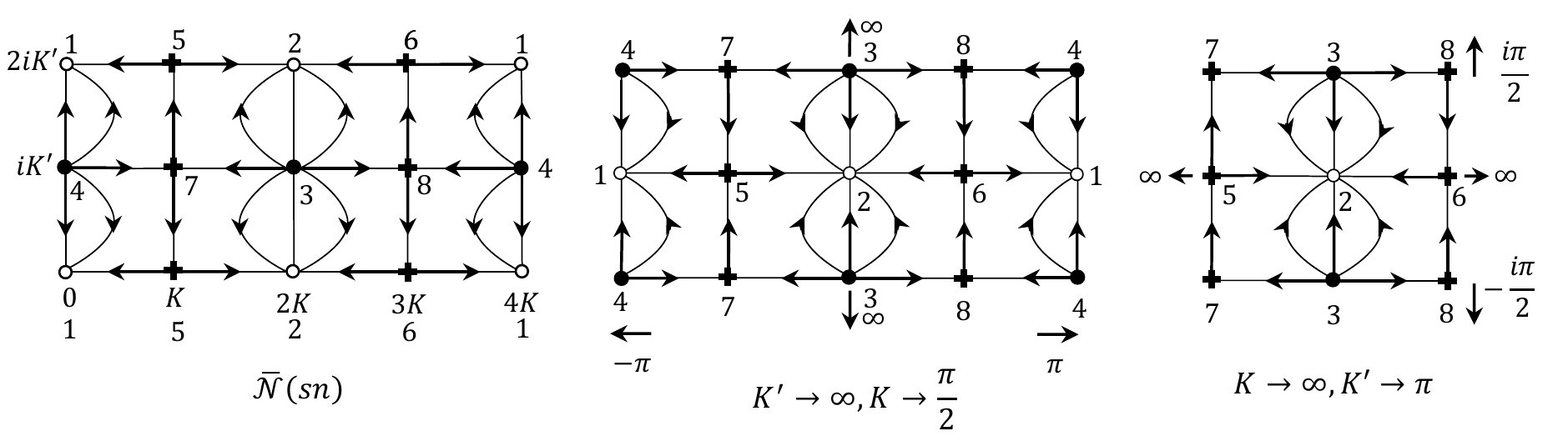}
\caption{ Newton flows for ${\rm sn}$; rectangular case; Re $\tau$=0}
\label{Figure10}
\end{figure}

\newpage
\section{Structural stability:  Characterization and Genericity}
\label{sec5}

Adopting the notations introduced in the preceding section, let $f$ be a function in the set $E_{r}$$(=E_{r}(\Lambda))$ and $\overline{\overline{\mathcal{N}}} (f)(=\overline{\overline{\mathcal{N}}} ([f])) \in N_{r}(=N_{r}(\Lambda))$ its associated Newton flow (as a smooth vector field on the torus $T(=T(\Lambda))$).

By $X(T)$
we mean the set of all $C^{1}$-vector fields on 
$T$, endowed with the $C^{1}$-topology (cf. \cite{Hir}). 
We consider the map: 
$$
\mathcal{F}_{\Lambda}: E_{r} \rightarrow X(T): f \mapsto \overline{\overline{\mathcal{N}}} (f)
$$

The topology $\tau_{0}$ on 
$E_{r}$ and the $C^{1}$-topology  on $X(T)$
are matched by: 
\begin{lemma}
\label{L5.1}
The map $\mathcal{F}_{\Lambda}$ is $\tau_{0}\!-\!C^{1}$ continuous.
\end{lemma}
\begin{proof}
In accordance with 
Theorem \ref{T4.6} and (\ref{vgl11}) we assume 
\begin{equation*}
f(z)= \frac{\sigma(z \!-\! a_1)\! \cdots \! \sigma (z\! -\! a_r)} {\sigma(z \!-\! b_1) \! \cdots \! \sigma(z\! \!-\! \!b_{r-1})  \sigma (z\! -\! b^{'}_r)}.
\end{equation*}
Put $p(z)=\sigma(z \!- \!a_1)\! \cdots \! \sigma (z - a_r)$ and $q(z)=\sigma(z \!-\! b_1) \! \cdots \! \sigma(z \!-\! b_{r-1})  \sigma (z \!- \!b^{'}_r)$. Then, the planar version $\overline{\mathcal{N}} (f)$ of the flow $\overline{\overline{\mathcal{N}}} (f)$ 
takes the form: (cf. (\ref{vgl3}))
\begin{equation}
\label{redNf}
\frac{dz}{dt}=-(|p(z)|^{4}+|q(z)|^{4})^{-1}(p(z)\overline{p^{'}(z)}|q(z)|^{2}-q(z)\overline{q^{'}(z)}|p(z)|^{2})
\end{equation}
The expression in the r.h.s. is well-defined (since $|p(z)|^{4}+|q(z)|^{4} \neq 0$ for all $z$) and depends -  as function ($F$) on $\mathbb{R}^{2}$- continuously differentiable on $x$(=Re $z$) and $y$(=Im $z$). So does the Jacobi matrix ($DF$) of $F$.  Analogously, a function 
$g \in E_{r}$ chosen $\tau_{0}$-close to $f$ , gives rise to a system $\overline{\mathcal{N}} (g)$ and a function $G$ with Jacobi matrix $DG$. Taking into account the very definition of $C^{1}$-topology on $X(T)$
, the mapping $\mathcal{F}_{\Lambda}$ is continuous as a consequence of the following observation: If -w.r.t. the topology $\tau_{0}$- the function $g$ approaches $f$ , i.e. the zeros and poles for $g$ approach those for $f$, then $G$ and $DG$ approach $F$, respectively $DF$ on every compact subset of $\mathbb{R}^{2}$.
\end{proof}

Next, we make the concept of structural stability for elliptic Newton flows precise.
\begin{definition}
\label{D5.2}
Let $f, g$ be two functions in $E_{r}$.
Then, the associated Newton flows are called {\it conjugate}, denoted 
$\overline{\overline{\mathcal{N}}} (f) \sim \overline{\overline{\mathcal{N}}} (g)$
, if there is a homeomorphism from $T$ 
onto itself, mapping maximal trajectories of 
$\overline{\overline{\mathcal{N}}} (f)$ onto those of $\overline{\overline{\mathcal{N}}} (g)$
, thereby respecting the orientation of these trajectories. 
\end{definition}
Note that the above definition is compatible with the concept of ``
equivalent representations of elliptic Newton flows'' as introduced in Section 4; compare also (the comment on) Definition \ref{D3.2}.

\begin{definition}
\label{D5.3}
The flow $\overline{\overline{\mathcal{N}}} (f) $
in $N_{r}$ 
is called {\it $\tau_{0}$-structurally stable} if there is a $\tau_{0}$-neighborhood $\mathcal{O}$ of $f$, such that for all $g \in \mathcal{O}
$ we have: $\overline{\overline{\mathcal{N}}} (f) \sim \overline{\overline{\mathcal{N}}} (g)$. 
\newline
The set of all structurally stable Newton flows $\overline{\overline{\mathcal{N}}} (f)$
 is denoted  $\tilde{N}_{r}$.
\end{definition}
From Lemma \ref{L5.1} it follows:
\begin{corollary}
\label{C5.4}
If $\overline{\overline{\mathcal{N}}} (f)$, as an element of $\mathcal{X}(T)$, is $C^{1}$-structurally stable (\cite{Peix1})
, then this flow
is also $\tau_{0}$-structurally stable. 
\end{corollary}
So, when discussing structural stability in the case of elliptic Newton flows, we may skip the adjectives $C^{1}$ and $\tau_{0}$.
\begin{definition}
\label{D5.5}
The function $f$ in $E_{r}$
 is called {\it non-degenerate} if:
\begin{itemize}
\item All zeros, poles and critical points for $f$ are simple;  
\item No two critical points for $f$
are connected by an
$\overline{\overline{\mathcal{N}}} (f)$-trajectory.
\end{itemize}
The set of all non degenerate functions in $E_{r}$ 
is denoted by $\tilde{E}_{r}$.
\end{definition}
\noindent
Note: If $f$ is non-degenerate, then $\dfrac{1}{f}$ also, and these functions share their critical points; moreover, 
   the derivative $f'$ is elliptic of order $2r$ (=number, counted by multiplicity, of the poles for $f$ on $T$).
  Since all zeros for $f$ are simple, the $2r$ zeros for $f'$ (on $T$) are the critical points (all simple) for $f$, and
  we find that $f$, as a function on $T$, has precisely $2r$ different critical points. Compare also the
  forthcoming Lemma \ref{L5.7} (Case $A=B=r$).\\

\noindent 
The main result of this section is:
\begin{theorem} $($Characterization and genericity of structural stability.$)$
\label{T5.6} 
\begin{enumerate}
\item  
$\overline{\overline{\mathcal{N}}} (f) $ is 
structurally stable 
if and only if  
$f $ in $\tilde{E}_{r}$.
\item The set 
$\tilde{E}_{r}$ is open and dense
in $E_{r}$.
\end{enumerate}
\end{theorem}
\begin{proof}
Will be postponed until the end of this section.   
\end{proof}
We choose another lattice, say $\Lambda^{*}$. The functions $f$ and $g $ in 
$E_{r}(\Lambda)$ 
determine respectively, 
functions $f^{*}$ and $ g^{*} $ in $E_{r}(\Lambda^{*})$, compare Lemma \ref{L4.4}. 
It is easily verified:
\begin{itemize}
\item 
$\overline{\overline{\mathcal{N}}} (f) \sim \overline{\overline{\mathcal{N}}} (g)$ if and only if  $\overline{\overline{\mathcal{N}}} (f^{*}) \sim \overline{\overline{\mathcal{N}}} (g^{*})$
\item 
$\overline{\overline{\mathcal{N}}}(f) \;$ is structurally stable if and only if  $\; \overline{\overline{\mathcal{N}}}(f^{*})$ is structurally stable.
\item $f$ is non-degenerate  
if and only if $f^{*} $ is non-degenerate.
\end{itemize}
\noindent
As an intermezzo, we look at (elliptic) Newton flows from a slightly different point of view. This enables us to perform certain calculations inserting more specific properties of elliptic functions.\\

\noindent
\underline{Steady streams}\\ 
We consider
a {\it stream} on $\mathbb{C}$ (cf. \cite{M1}) with complex potential
\begin{equation}
\label{vgl14}
w(z)=-\log f(z).
\end{equation}
The {\it stream lines} are given by the lines ${\rm arg }f(z)=$ constant, and the {\it velocity field} of this stream by $\overline{w^{'}(z)}$.
Zeros  and poles for $f$ of order $n$ resp. $m$ 
, are just the {\it sinks} and {\it sources} of strength $n$, resp. $m$. Moreover, it is easily verified that the so called {\it stagnation points} of the steady stream (i.e., the zeros for $\overline{w^{'}(z)}$) are the critical points of the planar Newton flow 
$\mathcal{N}(f)$. Altogether, we may conclude that the velocity field of the steady stream given by $w(z)$ and the (desingularized) planar Newton flow $\mathcal{N}(f)$ exhibit equal phase portraits.
\\ \\
From now on, we assume that $f$ has -on a period parallelogram $P $
- the points $({\bf a_{1}}, \! \cdots \!, {\bf a_{A}})$ and $({\bf b_{1}}, \! \cdots \!, {\bf b_{B}})$ as zeros, resp. poles, with multiplicities $n_{1}, \! \cdots \!, n_{A}$, resp. $m_{1}, \! \cdots \!, m_{B}$ We even may assume\footnote{If this is not the case, an (arbitrary small) shift of $P$ along its diagonal, is always possible such that the resulting parallelogram satisfies our assumption (cf. Fig. \ref{Figure11} and \cite{M2}). }  that all these zeros and poles are situated inside $P$ (not on its boundary), cf. Fig. \ref{Figure11}.

\begin{figure}[h]
\begin{center}
\includegraphics[scale=0.88]{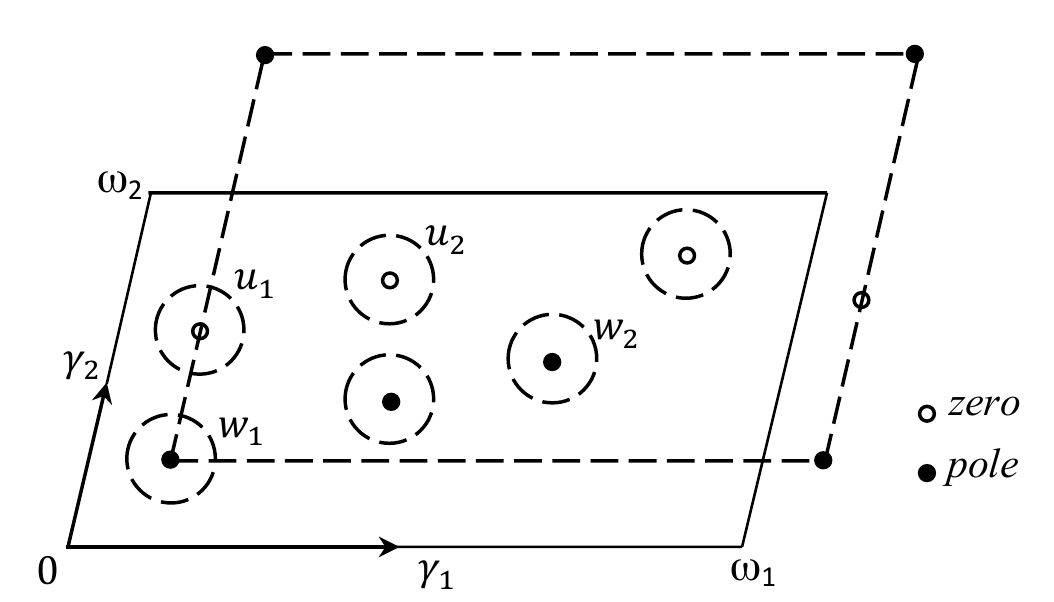}
\caption{\label{Figure11} All zeros and poles for $f$ inside $P$
; after shift}
\end{center}
\end{figure}

Since $f$ is elliptic of order $r$, we have:
\begin{equation}
\label{vgl15}
\begin{cases}
&n_{1} + \! \cdots \! + n_{A}=m_{1} + \! \cdots \! + m_{B}=r\\
& n_{1} {\bf a_{1}} + \! \cdots \! + n_{A} {\bf a_{A}}=m_{1}{\bf b_{1}} + \! \cdots \! + m_{B}{\bf b_{B}} \text{ mod } \Lambda 
 \text{, i.e.}\\
&{\bf b_{B}}=\frac{1}{m_{B}}[ n_{1} {\bf a_{1}} + \! \cdots \! + n_{A} {\bf a_{A}}-m_{1}{\bf b_{1}} - \! \cdots \! - m_{B-1}{\bf b_{B-1}} + \lambda^{0}], 
\text{ some }
\lambda^{0} \in \Lambda.
\end{cases}   
\end{equation} 
Note that there is an explicit formula for $\lambda^{0}$. In fact, we have:  
$$\lambda^{0}= -\eta(f(\gamma_{2}))\omega_{1}+ \eta(f(\gamma_{1}))\omega_{2},$$ where $\omega_{1}(=1)$ and  $\omega_{2}(= i)$ are reduced periods for $\Lambda$, and $\eta(\cdot)$ stands for winding numbers of the curves $f(\gamma_{1})$ and $f(\gamma_{2})$ (compare Fig. \ref{Figure11} and \cite{M2}).                                                                                                            

The derivative $f^{'}$ is an elliptic function of order $(m_{1}+1)+ \! \cdots \! +(m_{B}+1)=r+B$, and
$$ 
\sharp \text{(crit. points for } f)=r+B-(n_{1}-1)- \! \cdots \! -(n_{A}-1)=A+B(=:K)
$$
By (\ref{vgl11}) we have:
\begin{align}
\notag 
&f(z)= \frac{\sigma^{n_{1}}(z - {\bf a_{1}})\! \cdots \! \sigma^{n_{A}} (z - {\bf a_{A}})} {\sigma^{m_{1}}(z - {\bf b_{1}} ) \! \cdots \! \sigma^{m_{B}-1}(z - {\bf b_{B}} )  \sigma (z - {\bf b^{'}_{B}})}, \text{ with }\\ \label{vgl16}
&{\bf b^{'}_{B}}=n_{1} {\bf a_{1}} + \! \cdots \! + n_{A} {\bf a_{A}}-m_{1}{\bf b_{1}} - \! \cdots \! - (m_{B}-1){\bf b_{B}}, \\ \notag
& \text{ and thus }{\bf b^{'}_{B}}={\bf b_{B}} \text{ mod} \; \Lambda.
\end{align}
So, by (\ref{vgl14}) 
\begin{align*}
w'(z)=&-n_{1} \frac{\sigma'(z - {\bf a_{1}})}{\sigma(z - {\bf a_{1}})}- \! \cdots \! -n_{A} \frac{\sigma'(z - {\bf a_{A}})}{\sigma(z - {\bf a_{A}})}+m_{1} \frac{\sigma'(z - {\bf b_{1}})}{\sigma(z - {\bf b_{1}})} +\! \cdots \!\\
&\;+(m_{B}\!-\!1) \frac{\sigma'(z \!-\! {\bf b_{B}})}{\sigma(z \!-\! {\bf b_{B}})} +\frac{\sigma'(z \!-\! {\bf b^{'}_{B}})}{\sigma(z \!-\! {\bf b^{'}_{B}})}\\
=&-n_{1}\zeta(z \!-\! {\bf a_{1}}) \! \cdots \! -n_{A}\zeta(z \!-\! {\bf a_{A}})+m_{1}\zeta(z \!-\! {\bf b_{1}}) \! \cdots \! + (m_{B}\!-\!1)\zeta(z\! -\! {\bf b_{B}})+\zeta(z \!- \!{\bf b^{'}_{B}})
\end{align*}
where $\zeta$ stands for the WeierstrassÕ zeta function (cf. \cite{M2}) associated with the lattice $\Lambda$.
Since $\zeta$ is a quasi-periodic, meromorphic function with only poles (all simple!) in the points of  the lattice $\Lambda$, the function $w'(z)$ is elliptic of order $K(=A+B)$ with  $K$ (simple) poles given by: ${\bf a_{1}}, \! \cdots \! , {\bf a_{A}}, {\bf b_{1}}, \! \cdots \! , {\bf b_{B}}$
, situated in its period parallelogram $P$. It follows that $w'(z)$ has also $K$ zeros (counted by multiplicity) on $P$. These zeros correspond with the critical points for $\overline{\overline{\mathcal{N}}} (f)$ on the torus $T$.

Since $\zeta'\!=\!-\wp$, where $\wp$ stands for the (elliptic!) Weierstrass $\wp$-function (cf.\;\cite{M2}) associated with $\Lambda$
, we find 
(use also ${\bf b^{'}_{B}}={\bf b_{B}} \text{ mod }\Lambda$):
\begin{equation}
\label{vgl17}
w''(z)=n_{1}\wp(z \!-\! {\bf a_{1}}) \! \cdots \! +n_{A}\wp(z \!-\! {\bf a_{A}})-m_{1}\wp(z\! -\! {\bf b_{1}}) \! \cdots \! -m_{B}\wp(z \!-\! {\bf b_{B}})
\end{equation}
{}From (\ref{vgl15}) and (\ref{vgl16}) it follows: ${\bf b_{B}}$ is determined by ${\bf a_{1}}, \! \cdots \! , {\bf a_{A}}, {\bf b_{1}}, \! \cdots \! , {\bf b_{B-1}}$. In mutually disjoint and suitably small\footnote{Choose these neighborhoods so that they are contained in the period parallelogram $P$, cf Fig. \ref{Figure11}.}  neighborhoods, say $U_{1}, \! \cdots \! , U_{A},$ and $W_{1}, \! \cdots \! , W_{B-1}$, of respectively ${\bf a_{1}}, \! \cdots \! , {\bf a_{A}}, {\bf b_{1}}, \! \cdots \! , {\bf b_{B-1}}$ we choose arbitrary points $a_{1}, \! \cdots \!, a_{A}, b_{1}, \! \cdots \!,  b_{B-1}$ and put, with fixed values $n_{i}, m_{j}$ : 
\begin{equation}
\label{vglbB}
b_{B}=\frac{1}{m_{B}}[n_{1}a_{1}+ \! \cdots \! + n_{A}a_{A} - m_{1}b_{1}- \! \cdots \! -m_{B-1}b_{ B-1}+ \lambda^{0}]. 
\end{equation}
In this way $a_{1}, \! \cdots \!, a_{A}, b_{1}, \! \cdots \!,  b_{B-1}$ are close to respectively ${\bf a_{1}}, \! \cdots \! , {\bf a_{A}}, {\bf b_{1}}, \! \cdots \! , {\bf b_{B-1}}$ and $b_{B}$ is close to ${\bf b_{B}}$. Finally, we put 
\begin{equation}
\label{vglbBacc}
b^{'}_{B}=n_{1}a_{1}+ \! \cdots \! + n_{A}a_{A} - m_{1}b_{1} \! \cdots \! -m_{B-1}b_{ B-1}-(m_{B}-1)b_{B}\;\;\; \text{ (close to ${\bf b^{'}_{B}}$).}
\end{equation}
Perturbating ${\bf a_{1}}, \! \cdots \! , {\bf a_{A}}, {\bf b_{1}}, \! \cdots \! , {\bf b_{B-1}}$  into respectively $a_{1}, \! \cdots \!, a_{A}, b_{1}, \! \cdots \!,  b_{B-1}$ , and putting 
 $\check{a}=(a_{1}, \! \cdots \!, a_{A}), \check{b}=(b_{1}, \! \cdots \!,  b_{B-1})$
, we consider  functions\footnote{Note that, when perturbing (${\bf a,b}$) in the indicated way, the winding numbers $\eta({\bf f}_{(\cdot, a, b)}(\gamma_{1}))$, $\eta ({\bf f}_{(\cdot, a, b)} (\gamma_{2}))$, and thus also $\lambda^{0}$, remain unchanged.} 
${\bf f}(z; \check{a},\check{b})$ on the product space $\mathbb{C} \times U_{1} \times \! \cdots \! U_{A} \times W_{1} \times \! \cdots \! W_{B-1}$
given by: (compare (\ref{vgl16})) 
\begin{equation}
\label{vgl18}
{\bf f}(z; \check{a},\check{b})=\frac{\sigma^{n_{1}}(z - a_{1})\! \cdots \! \sigma^{n_{A}} (z - a_{A})} {\sigma^{m_{1}}(z - b_{1}) \! \cdots \! \sigma^{m_{B}-1}(z - b_{B}(\check{a},\check{b}) )  \sigma (z - b^{'}_{B}(\check{a},\check{b}))}. 
\end{equation}
Then, for each $(\check{a},\check{b})$ in $U_{1} \times \! \cdots \! U_{A} \times W_{1} \times \! \cdots \! W_{B-1}$
, the function $
f|_{\check{a},\check{b}}(\cdot)={\bf f}(\cdot ;\check{a},\check{b})$
is elliptic in $z$, of order $r$.
The points $a_{1}, \! \cdots \!, a_{A}$, resp. $b_{1}, \! \cdots \!,  b_{B}$  are  the zeros and poles for $ f|_{\check{a},\check{b}}$ on $P$ ( of multiplicity $n_{1}, \! \cdots \!,  n_{A},$ resp. $m_{1}, \! \cdots \! , m_{B})$. Moreover, $f|_{\check{a},\check{b}}$ has $K(=A+B)$ critical points on $P$ (counted by multiplicity). 
Note that the Newton flow $\overline{\overline{\mathcal{N}}} (f|_{\check{a},\check{b}})$ on $T$ is represented 
by the pair $(\check{a},\check{b})$ in $U_{1} \times \! \cdots \! U_{A} \times W_{1} \times \! \cdots \! W_{B-1}$, i.e. by
$$
\begin{matrix}
a_{1},& \!\! \cdots \! ,& a_{A};& b_{1},& \!\! \cdots \! ,& b_{B-1},&\text{ arbitrarily chosen in suitably small}\\
\uparrow & &\uparrow &\uparrow & & \uparrow& \text{neighbourhoods $U_{1} \times \! \cdots \! U_{A} \times W_{1} \times \! \cdots \! W_{B-1}$}
\\
1 \times& &1 \times & 1 \times & &1 \times  &
\end{matrix}
$$
but also 
by the pair $(a, b)$ in the quotient space $V_{r}(\Lambda)/ \approx$  as introduced in Section \ref{sec4}, i.e. by
$$
\begin{matrix}
(
([a_{1}], &\!\! \cdots \! ,&[a_{A}] )
,&
([b_{1}], &\!\! \cdots \! ,&[b_{B}] )
), & \text{ that fulfil condition (\ref{vgl10})}\\
\uparrow & & \uparrow&\uparrow & &\uparrow &\\
n_{1} \times & &n_{A} \times & m_{1} \times& & m_{B} \times&
\end{matrix}
$$
In particular, 
\begin{itemize}
\item 
If  $(\check{a},\check{b})=(\check{{\bf a}},\check{{\bf b}})$, then  $f|_{\check{{\bf a}},\check{{\bf b}}}(z)=f(z)$ and thus $\overline{\overline{\mathcal{N}}} (f|_{\check{{\bf a}},\check{{\bf b}}})=\overline{\overline{\mathcal{N}}} (f)$;
\item
If $A=B=r$ (thus $K=2r$), i.e., all zeros, poles are simple, then 
\begin{align*}
&(\check{a},\check{b})=(a_{1}, \! \cdots \!, a_{r}; b_{1}, \! \cdots \! , b_{r-1}) \in U_{1} \times \! \cdots \! U_{r} \times W_{1} \times \! \cdots \! W_{r-1}, \text{ and}\\
&(a,b)=((
[a_{1}], \! \cdots \! ,[a_{r}] 
), 
([b_{1}], \! \cdots \! ,[b_{r}] )
)+ \text{ condition }
(\ref{vgl10})
\end{align*}
\end{itemize}
The decomposition of multiple critical points is by far the hardest part in the density proof of
Theorem 5.6 and strongly relies on the following lemma. Note that in case of {\it non-degenerate} functions this lemma is trivial (see the note on Page 13 between Definition 5.5 and Theorem 5.6). But we also we have to cope with {\it degenerate} functions $f$. So, we need an analysis of the influence of perturbations of the zeros/poles for such functions on their critical points (being implicitly determined by these zeros/poles). Here is the lemma we need:
\begin{lemma}
\label{L5.7}
Let $K(=A+B)>2$. Then:\\
\noindent
Under suitably chosen -but arbitrarily small- perturbations of the zeros and poles for $f$, thereby preserving the multiplicities of these zeros and poles, the Newton flow $\overline{\overline{\mathcal{N}}} (f))$ turns into a flow $\overline{\overline{\mathcal{N}}} (f|_{\check{a},\check{b}})$ with only ($K$ different)1-fold saddles.
\end{lemma}
\begin{proof}
We consider $\hat{w}(z; \check{a}, \check{b})= - {\rm log} {\bf f} (z; \check{a}, \check{b})$ and write: 
$$
\frac{\partial \hat{w}}{\partial z}=\hat{w}^{'}(z; \check{a}, \check{b});\, \,
\frac{\partial^{2} \hat{w}}{\partial^{2} z}=\hat{w}^{''}(z; \check{a}, \check{b}).
$$
These are meromorphic functions in each of the variables $z, a_{1}, \! \cdots \!, a_{A}, b_{1}, \! \cdots \!,  b_{B-1}$.\\
Define:
\begin{align*}
&\Sigma=\{ (z; \check{a}, \check{b}) \mid \hat{w}^{'}(z; \check{a}, \check{b})=0\} &\text{``critical set''}\\
& \Sigma_{nd}=\{ (z; \check{a}, \check{b}) \mid \hat{w}^{'}(z; \check{a}, \check{b})=0,\hat{w}^{''}(z; \check{a}, \check{b}) \neq 0\}&\text{``non-degenerate critical set''}\\
& \Sigma_{d}=\{ (z; \check{a}, \check{b}) \mid \hat{w}^{'}(z; \check{a}, \check{b})=0, \hat{w}^{''}(z; \check{a}, \check{b})=0\}&\text{``degenerate critical set''}\\
\end{align*}
Since the $a_{i}, i=1, \! \cdots \! , A,$ and $b_{j}, j=1, \! \cdots \! ,B$, are poles for $\hat{w}^{'}(\cdot; \check{a}, \check{b})$ as an elliptic function in $z$, we have: If $(z; \check{a}, \check{b}) \in \Sigma$, then:
\begin{equation*}
\begin{cases}
z \neq a_{i} \text{ mod }\Lambda, z \neq b_{j} \text{ mod }\Lambda, z \neq b^{'}_{B} \text{ mod }\Lambda, \text{ and (by construction)}\\
a_{i_{1}} \neq a_{i_{2}} \text{ mod }\Lambda, i_{1} \neq i_{2}, i_{1}, i_{2}=1, \! \cdots \! , A;b_{j_{1}} \neq b_{j_{2}} \text{ mod }\Lambda, j_{1} \neq j_{2}, j_{1}, j_{2}=1, \! \cdots \! , B\!-\!1.
\end{cases}
\end{equation*}
The subset $\mathcal{V}$ of elements $(z; \check{a}, \check{b}) \in  \mathbb{C} \times U_{1} \times \! \cdots \! U_{A} \times W_{1} \times \! \cdots \! W_{B-1}$ that fulfills these inequalities is open in $\mathbb{C} \times\mathbb{C}^{A} \times \mathbb{C}^{B-1}$. On this set $\mathcal{V}$ (that contains the critical set $\Sigma$), the function $\hat{w}^{'}$ is analytic in each of its variables. (Thus $\Sigma$ is a closed subset of $\mathcal{V}$).
For the partial derivatives of $\hat{w}^{'}$ on $\mathcal{V}$ we find: (use (\ref{vgl14}), (\ref{vgl18}), compare also (\ref{vgl17}) and Footnote 9)
\begin{align*}
\frac{\partial \hat{w}^{'}}{\partial z}:\;&n_{1} \wp (z-a_{1}) + \! \cdots \! +n_{A} \wp (z-a_{A})-m_{1} \wp (z-b_{1})- \! \cdots \! -m_{B} \wp (z-b_{B})\\
\frac{\partial \hat{w}^{'}}{\partial a_{i}}:\;&\frac{\partial}{\partial a_{i}}[-n_{1} \zeta (z\!-\!a_{1}) \! \cdots \! -n_{A} \zeta (z\!-\!a_{A})+m_{1} \zeta (z-b_{1}) \! \cdots \! +(m_{B}-1) \zeta (z-b_{B}) + \zeta(z-b^{'}_{B})]\\
 &=[ \text{ using the formulas (\ref{vglbB}) and (\ref{vglbBacc}})] \\
 &=[ -n_{i} \wp (z-a_{i}) +(m_{B}-1)[\frac{n_{i}}{m_{B}}] \wp (z-b_{B})+(n_{i}-(m_{B}-1)[\frac{n_{i}}{m_{B}}]) \wp (z-b^{'}_{B})]\\
 &=-n_{i}(\wp (z-a_{i}) -\wp (z-b_{B}) ) \;\;\;\;\;\; \;\;\; (i=1, \! \cdots \! , A)\\
\text{In a }& \text{similar way:}\\
\frac{\partial \hat{w}^{'}}{\partial b_{j}}: \;& m_{j}(\wp (z-b_{i}) -\wp (z-b_{B}) ) \;\;\;(j=1, \! \cdots \! , B-1)\\
[\frac{\partial \hat{w}^{'}}{\partial b_{B}}&=0 \;\;\;(m_{B}=1, 2, \! \cdots \! )]
\end{align*}
By the Addition Theorem of the $\wp$-function [cf. \cite{M2}], we have:
$$
\frac{\partial \hat{w}^{'}}{\partial a_{i}}: - \frac{\sigma(a_{i}\!-\!b_{B})\sigma(2z\!-\!a_{i}-b_{B})}{\sigma^{2}(z\!-\!a_{i})\sigma^{2}(z\!-\!b_{B})}, i=1, \! \cdots \! , A
$$
So, let $(z; \check{a}, \check{b}) $ in $\mathcal{V}$, then
\begin{equation*}
\frac{\partial \hat{w}^{'}}{\partial a_{i}}|_{(z; \check{a}, \check{b})}=0
, \text{ some }i \in \{1, \! \cdots \! , A \} 
\Leftrightarrow
\end{equation*}
$$
\begin{cases}
a_{i}=b_{B}\text{ mod }\Lambda,& \text{[in contradiction with ``$a_{i}, b_{B}$ different'']}\\
\text{ or }&\\
2z=a_{i}+b_{B}\text{ mod }\Lambda & \text{[ if $2z=a_{i_{1}}+b_{B}\text{ mod }\Lambda$, $2z=a_{i_{2}}+b_{B}\text{ mod }\Lambda,i_{1} \neq i_{2},$}\\
&\text{then $a_{i_{1}}=a_{i_{2}}$; in contradiction with ``$a_{i_{1}}, a_{i_{2}}$ different'']}
\end{cases}
$$
From this follows:
\\ \\
\noindent
If $(z; \check{a}, \check{b}) \in \mathcal{V}$ , then at most one of  $\frac{\partial \hat{w}^{'}}{\partial a_{i}}|_{(z; \check{a}, \check{b})},i =1, \! \cdots \! , A,  $ vanishes. 
By a similar reasoning, we even may conclude:
\begin{equation}
\label{vgl19}
\begin{cases}
\text{At most one of the partial derivatives} \\
$$
\frac{\partial \hat{w}^{'}}{\partial a_{i}}(z; \check{a}, \check{b}),\frac{\partial \hat{w}^{'}}{\partial b_{j}}(z; \check{a}, \check{b}), (z; \check{a}, \check{b}) \in \mathcal{V}, i =1, \! \cdots \! , A,  j=1, \! \cdots \! , B-1,
$$
\\
\text{vanishes, and thus, {\bf in case ${\bf K>2}$:} }\\
$$
\frac{\partial \hat{w}^{'}}{\partial a_{i}}(z; \check{a}, \check{b}) \neq 0, \frac{\partial \hat{w}^{'}}{\partial b_{j}}(z; \check{a}, \check{b}) \neq 0, \\
\text{ for at least one $i \in \{1, \! \cdots \! , A \}$ or $j \in \{1, \! \cdots \! , B-1. \}$}
$$
\end{cases}
\end{equation}
The latter conclusion cannot be drawn in case $K=2$; however, see the forthcoming Remark \ref{R5.8}. Note that always $K\geqslant2$. 
\\ \\                                                                                                
\noindent
Under the assumption that $K >2: $
let ${\bf z_{1}},{\bf z_{2}}, \! \cdots \! , {\bf z_{L}}$
be the different critical points for $f=({\bf f}(\cdot, \check{{\bf a}}, \check{{\bf b}}))$ 
with multiplicities $K_{1}, \! \cdots \! , K_{L}, K_{1} \geqslant \! \cdots \! \geqslant K_{L} \geqslant 1, K_{1}+ \! \cdots \! +K_{L}=K$.
If $(\check{a}, \check{b})$ tends to $({\bf \check{a}}, {\bf \check{b}})$, then $K_{l}$ of the $K$ critical points for ${\bf f}(\cdot, \check{a}, \check{b})$(counted by multiplicity) tend to the $K_{l}$-fold saddle ${\bf z_{l}}$ for $\overline{\overline{\mathcal{N}}} (f))$. It follows that, if $(\check{a}, \check{b})$ is sufficiently  close to $({\bf \check{a}}, {\bf \check{b}})$, then $K_{l}$ critical points for ${\bf f}(\cdot, \check{a}, \check{b})$(counted by multiplicity) are situated  in, suitably small, disjunct neighborhoods, say $V_{l}$, around ${\bf z_{l}}, l=1, \! \cdots \! ,L.$
We choose $(\check{a}, \check{b})$ so close to $({\bf \check{a}}, {\bf \check{b}})$  that this condition holds. 
If all the critical points for $f$, i.e. the saddles of $\overline{\overline{\mathcal{N}}} (f))$, are simple, there is nothing to prove. 
So, let $K_{1}>1$, thus $\hat{w}^{''}({\bf z_{1}}; {\bf \check{a}}, {\bf \check{b}})=0$, i.e. $({\bf z_{1}}; {\bf \check{a}}, {\bf \check{b}}) \in \Sigma_{d} \subset \Sigma$. Without loss of generality, we  assume (see (\ref{vgl19})) that $\frac{\partial \hat{w}^{'}}{\partial a_{i}}({\bf z_{1}}; {\bf \check{a}}, {\bf \check{b}}) \neq 0$. According to the Implicit Function Theorem a  local parametrization of $\Sigma$ around $({\bf z_{1}}; {\bf \check{a}}, {\bf \check{b}})$ exists, given by:  
$$
(z; a_{1}(z, a_{2}, \! \cdots \! ,a_{A}, b_{1}, \! \cdots \! , b_{B-1}), a_{2}, \! \cdots \! ,a_{A}, b_{1}, \! \cdots \! , b_{B-1}), 
$$
where $a_{1}({\bf z_{1}}, {\bf a_{2}}, \! \cdots \! ,{\bf a_{A}}, {\bf b_{1}}, \! \cdots \! , {\bf b_{B-1}})={\bf a_{1}}.$ Thus, at $({\bf z_{1}}; {\bf \check{a}}, {\bf \check{b}})$ we have: 
$$
\hat{w}^{''}+[\frac{\partial \hat{w}^{'}}{\partial a_{1}}][\frac{\partial a_{1}}{\partial z}(z,a_{2}, \! \cdots \!)]=0.
$$
Since $\hat{w}^{''}({\bf z_{1}}; {\bf \check{a}}, {\bf \check{b}})=0$ and $\frac{\partial \hat{w}^{'}}{\partial a_{i}}({\bf z_{1}}; {\bf \check{a}}, {\bf \check{b}}) \neq 0$, it follows that
$$
\frac{\partial a_{1}}{\partial z}({\bf z_{1}},{\bf \check{a}}, {\bf \check{b}})=0
$$
Note that $a_{1}(z, a_{2}, \! \cdots \! ,a_{A}, b_{1}, \! \cdots \! , b_{B-1})$, depends complex differentiable on $z$. So the zeros for 
$\frac{\partial a_{1}}{\partial z}(z;a_{2}, \! \cdots \!, a_{A}, b_{1}, \! \cdots \! , b_{B-1})$ are {\it isolated}. Thus, on a {\it reduced} neighborhood of $({\bf z_{1}},{\bf \check{a}}, {\bf \check{b}})$, say $\hat{U}$, neither $\frac{\partial a_{1}}{\partial z}(\cdot)$ nor $\frac{\partial \hat{w}^{'}}{\partial a_{1}}$ vanish. If $z$ tends to ${\bf z_{1}}$ , then: 
$$
(z; a_{1}(z, {\bf a_{2}}, \! \cdots \! ,{\bf a_{A}}, {\bf b_{1}}, \! \cdots \! , {\bf b_{B-1}}), {\bf a_{2}}, \! \cdots \! ,{\bf a_{A}}, {\bf b_{1}}, \! \cdots \! , {\bf b_{B-1}})
$$
tends to $({\bf z_{1}},{\bf \check{a}}, {\bf \check{b}})$ along $\Sigma$, and we cross $\hat{U}$, meeting elements $(z, \check{a}, \check{b}) \in \Sigma$, such that
\begin{equation*}
\begin{cases}
&\hat{w}^{''}(z, \check{a}, \check{b})+\frac{\partial \hat{w}^{'}}{\partial a_{1}}(z, \check{a}, \check{b}) \frac{\partial a_{1}}{\partial z} (z, \check{a}, \check{b}) =0\\
&\frac{\partial \hat{w}^{'}}{\partial a_{1}}(z, \check{a}, \check{b}) \neq 0, \frac{\partial a_{1}}{\partial z} (z, \check{a}, \check{b}) \neq 0
\end{cases}
\end{equation*}
Thus,
$$
\hat{w}^{''}(z, \check{a}, \check{b}) \neq 0.
$$
Hence, we have: $(z, \check{a}, \check{b}) \in \Sigma_{nd}$.
So, the $K_{1}$ critical points for ${\bf f}(\cdot ; \check{a}, \check{b})$ that approach ${\bf z_{1}}$ via the ÒcurveÓ 
$$
(z; a_{1}(z, {\bf a_{2}}, \! \cdots \! ,{\bf a_{A}}, {\bf b_{1}}, \! \cdots \! , {\bf b_{B-1}}), {\bf a_{2}}, \! \cdots \! ,{\bf a_{A}}, {\bf b_{1}}, \! \cdots \! , {\bf b_{B-1}})
$$
are all simple, whereas the critical points for ${\bf f}(\cdot, \check{a}, \check{b})$ approaching ${\bf z_{2}}, \! \cdots \! ,{\bf z_{L}}$ are still situated in respectively $V_{2}, \! \cdots \! ,V_{L}$.
If $K_{2}>1$, we repeat the above procedure with respect to ${\bf z_{2}}$, etc.
In finitely many steps we arrive at a flow $\overline{\overline{\mathcal{N}}} ({\bf f}|_{\check{a}, \check{b}})$ with only simple saddles and $( \check{a}, \check{b})$ arbitrary close to $({\bf \check{a}}, {\bf \check{b}})$.
\end{proof}
\begin{remark}
\label{R5.8}
The case $A=B=1 (\text{ i.e. }K=2)$.\\
\noindent
If $K =2$, then the function $f$  has -on $T$- only one zero and one pole, both of order $r$; the corresponding flow $\overline{\overline{\mathcal{N}}} (f))$ is referred to as to a {\it nuclear} 
Newton flow.
In this case, the assertion of Lemma \ref{L5.7} is also true. In fact, even a stronger result holds:  
\begin{align*}
&\text{`` All nuclear Newton flows -of any order $r$- are conjugate
, in particular each of them}\\
&\text{ has precisely two saddles (simple) and there are no saddle connections''.}
\end{align*}
Nuclear Newton flows will play an important role in the creation of elliptic Newton flows, but we postpone the discussion on this subject to a sequel of the present paper, see \cite{HT4}.
\end{remark}
\noindent
We end up by presenting the (already announced) proof of Theorem \ref{T5.6}\\ 

{\it Proof of Theorem \ref{T5.6}: }\\
\noindent
\underline{The ``density part'' of Assertion 2}: Let $\mathcal{O}$ be an arbitrarily small $\tau_{0}$-neighbourhood of  a function $f $ as in Lemma \ref{L5.7}.
We split up each of the $A$ different zeros and $B$ different poles for $f$ into $n_{i}$ resp. $m_{j}$ mutually different points, contained in disjoint neighbourhoods $U_{i}$ resp. $W_{j}, i=1, \! \cdots \! , A, j=1, \! \cdots \! , B$ (compare Fig.\ref{Figure11}) and take into account relation (\ref{vgl9}).
In this way, we obtain 2$r$ different points, giving rise to an elliptic function of the the form (\ref{vgl11}), with these points as the  $r$ simple zeros 
/$r$ simple poles in $P$. We may assume that this function is still 
situated in $\mathcal{O}$, see the introduction of the topology $\tau_{0}$ in Section \ref{sec4}. Now, we apply Lemma \ref{L5.7} (case $ A=B=r, K=2r$) and find in $\mathcal{O}$ an elliptic function, of order $r$ with only simple zeros, poles and critical points. This function is non-degenerate if the corresponding Newton flow does not exhibit trajectories that connect two of its critical points. If this is the case, none of the straight lines connecting two critical values for our function, passes through the point $0 \in \mathbb{C}$. If not, then adding an arbitrarily small constant $c \in \mathbb{C}$ to $f$ does not affect the position of its critical points, and yields a function -still in\footnote{Note that at simple zeros an analytic function is conformal. In case of a pole, use also (\ref{vgl6}).}  $\mathcal{O}$- with only simple zeros and poles. By choosing $c$ suitably, we find a function, renamed $f$,  such that  none of 
the straight lines connecting critical values (for different critical points) contains $0 \in \mathbb{C}$. So, we have: $f \in \tilde{E}_{r}$, i.e., $\tilde{E}_{r}$ is dense in $E_{r}$.\\

\noindent
\underline{The ``if part'' of Assertion 1}: Let $f \in \tilde{E}_{r}$. Then all equilibria for the flow $\overline{\overline{\mathcal{N}}} (f)$ are 
hyperbolic (cf. Remarks \ref{R1.1} and \ref{R3.4}). Moreover, there are neither saddle-connections nor closed orbits (compare (\ref{vgl4})) and the limiting sets of the trajectories are isolated equilibria (cf. Lemma \ref{L3.3}). Now, the Baggis-Peixoto Theorem for structurally stable $C^{1}$-vector fields on compact 2-dimensional manifolds (cf.\cite{Peix1}, \cite{Peix2})
yields that $\overline{\overline{\mathcal{N}}} (f)$ 
is $C^{1}$-structurally stable, and is by Corollary \ref{C5.4} also $\tau_{0}$-structural stable. \\

\noindent
\underline{The ``only if part'' of Assertion 1}: Suppose that $f \notin \tilde{E}_{r}$, but $\overline{\overline{\mathcal{N}}} (f)$ in $\tilde{N}_{r}$. Then there is a $\tau_{0}$-neighbourhood of $f$, say $\mathcal{O}$, such that for all $g \in \mathcal{O}$: $\overline{\overline{\mathcal{N}}} (f) \sim \overline{\overline{\mathcal{N}}} (g)$. Since $\tilde{E}_{r}$ is dense in $E_{r}$(already proved), we may assume that  $g \in \tilde{E}_{r}$. So, $\overline{\overline{\mathcal{N}}} (g)$ has precisely $r$ hyperbolic attractors/repellors and does not admit ``saddle connections''. This must also be true for $\overline{\overline{\mathcal{N}}} (f)$, in contradiction with $f \notin \tilde{E}_{r}$. \\

\noindent
\underline{The ``openess part'' of Assertion 2}: This a direct consequence of the Assertion 1 (which has already been proved).

\end{document}